\documentclass[reqno,11pt]{amsart}

\usepackage{a4wide}
\usepackage[english]{babel}
\usepackage[utf8]{inputenc}
\usepackage{amsmath,amssymb}
\usepackage{tikz}
\usepackage[hidelinks]{hyperref}

\newcommand{\dom}{\operatorname{dom}}
\newcommand{\ran}{\operatorname{ran}}
\newcommand{\Sc}{\operatorname{Sc}}
\renewcommand{\Im}{\operatorname{Im}}
\newcommand{\sgn}{\operatorname{sgn}}
\newcommand{\Div}{\operatorname{div}}
\newcommand{\bnd}{{\text{bnd}}}
\newtheorem{thm}{Theorem}[section]
\newtheorem{lem}[thm]{Lemma}
\newtheorem{prop}[thm]{Proposition}
\newtheorem{cor}[thm]{Corollary}
\newtheorem{defi}[thm]{Definition}
\newtheorem{rem}[thm]{Remark}
\newtheorem{ass}[thm]{Assumption}
\numberwithin{equation}{section}

\begin{document}

\title[Nonlocal Fourier Laws for Heat Propagation]{Nonlocal Fourier Laws for Heat Propagation via Fractional powers of Vector Operators}

\author[F. Colombo]{Fabrizio Colombo}
\address{(FC) Politecnico di Milano, Dipartimento di Matematica, Via E. Bonardi 9, 20133 Milano, Italy}
\email{fabrizio.colombo@polimi.it}

\author[F. Mantovani]{Francesco Mantovani}
\address{(FM) Politecnico di Milano, Dipartimento di Matematica, Via E. Bonardi 9, 20133 Milano, Italy}
\email{francesco.mantovani@polimi.it}

\author[P. Schlosser]{Peter Schlosser}
\address{(PS) Institute of Applied Mathematics, Graz University of Technology, Steyrergasse 30, 8010 Graz, Austria}
\email{pschlosser@math.tugraz.at}

\thanks{F. Colombo is supported by MUR grant Dipartimento di Eccellenza 2023-2027.}

\begin{abstract}
The present work is devoted to the study of fractional powers of vector operators, with particular emphasis on the gradient operator with non-constant coefficients. Within the setting of Clifford algebra $\mathbb{R}_n$, this operator turns out to have bisectorial properties.
By applying the spectral theory on the $S$-spectrum, we address a fundamental mathematical challenge: unlike sectorial operators, bisectorial operators involve fractional powers that are not analytic on the negative real line. To circumvent this, we introduce a novel definition of the fractional power function in this setting. Building upon previous works on bisectorial vector operators and weak solutions, we extend the definition of fractional powers to abstract vector operators. The core contribution of this work is the application of the functional calculus for vector operators to the gradient operator, showing that these fractional powers provide a rigorous mathematical foundation for nonlocal Fourier laws in heat propagation.
\end{abstract}

\maketitle

AMS Classification 47A10, 47A60. \medskip

Keywords: Nonlocal Fourier laws, Fractional powers, Vector operator, Generalized Gradient, $S$-spectrum.

\section{Introduction}

In this paper we study the gradient operator with non-constant
coefficients $a_i\colon\mathbb{R}^n\to\mathbb{R}$, acting within the Clifford algebra
$\mathbb{R}_n$ of dimension $n\geq 3$.
Specifically, let $e_1,\dots,e_n$ denote the standard imaginary units of
$\mathbb{R}_n$ satisfying the anticommutation relations
$e_ie_j+e_je_i=-2\delta_{ij}$;
the gradient operator is then formally defined as

\begin{equation}\label{Eq_Gradient_formal}
  \nabla_a := \sum_{i=1}^n e_i\,a_i(x)\,\frac{\partial}{\partial x_i}.
\end{equation}
Applying the spectral theory on the $S$-spectrum, we use the
$H^\infty$-functional calculus associated with the operator $\nabla_a$
in the Clifford-algebra setting.
This functional calculus provides a rigorous framework for defining the
fractional powers of vector operators such as $\nabla_a$.

The key observation is that the fractional powers
of $\nabla_a$, constructed via the $H^\infty$-functional calculus
on the $S$-spectrum, provide a natural operator-theoretic model for the
nonlocal Fourier law.
This constitutive relation generalises Fourier's classical local law
to a setting in which heat conduction at a given point depends on the
temperature distribution over an extended spatial region.
In particular, when all coefficients are constant and equal to one,
i.e., $a_i\equiv 1$ for every $i=1,\dots,n$, the operator $\nabla_a$
reduces to the standard gradient
\[
  \nabla := \sum_{i=1}^n e_i\,\frac{\partial}{\partial x_i},
\]
and the fractional powers of $\nabla$ constructed within the
proposed framework, when we apply the divergence operator, recover the classical fractional Laplacian. This consistency result confirms that the theory developed here is a
genuine extension of the well established constant coefficient theory to
the variable coefficient in the Clifford algebra settings.

\medskip

Spectral analysis of operators in a Banach module over $\mathbb{R}_n$ is significantly more complicated than in the complex case, due to the non-commutativity of $\mathbb{R}_n$. The natural notion of spectrum in this setting is the $S$-spectrum, see Definition~\ref{defi_S_spectrum}, which is closely connected to hypercomplex analysis. For comprehensive discussions, the reader is referred to the books \cite{FJBOOK,CGK,ColomboSabadiniStruppa2011}. This spectral theory on the $S$-spectrum also applies to operators in Vector Analysis and Differential Geometry, particularly for the study of Dirac operators on manifolds, which are specific instances of Clifford operators. \medskip

This theory originated to provide a rigorous mathematical foundation for Quaternionic Quantum Mechanics, but it was subsequently generalized to the setting of Clifford operators. This foundational framework was established by Birkhoff and von Neumann \cite{Birkhoff1936} and expanded upon in works such as \cite{Emch1963,Finkelstein1962,Horwitz1984} and Adler's book \cite{Adler1995}. \medskip

In our previous work \cite{Gradient}, we study the $S$-spectral problem of the gradient operator \eqref{Eq_Gradient_formal} and concluded its bisectorial nature. In particular, $\nabla_a$ has $S$-spectrum on the negative as well as on the positive real line. This makes it hard to give meaning to fractional powers $\nabla_a^\alpha$ (where $\alpha$ is a suitable real number) of the gradient, since the classical fractional power function $s^\alpha$ is not defined (not holomorphic) on the negative real axis. Compared to complex fractional powers, in the hypercomplex setting of $\mathbb{R}_n$ it is also not possible to rotate this axis of non-holomorphicity into another direction, it is fixed on the negative real line. Our solution to this problem are the functions
\begin{equation}\label{Eq_palpha_qalpha}
p_\alpha(s):=\begin{cases} s^\alpha, & \Sc(s)>0, \\ -(-s)^\alpha, & \Sc(s)<0. \end{cases}\qquad\text{and}\qquad q_\alpha(s):=\begin{cases} s^\alpha, & \Sc(s)>0, \\ (-s)^\alpha, & \Sc(s)<0, \end{cases}
\end{equation}
where the precise definition of the term $s^\alpha$ is
\begin{equation}\label{Eq_salpha}
s^\alpha:=e^{\alpha\ln(s)},\qquad\Sc(s)>0,
\end{equation}
using the logarithm of a Clifford number
\begin{equation}\label{Eq_ln}
\ln(s):=\ln|s|+J\arg(s),\qquad\Sc(s)>0.
\end{equation}
Here $J\in\mathbb{S}$ is such that $s\in\mathbb{C}_J$, and $\arg(s)\in(-\frac{\pi}{2},\frac{\pi}{2})$ is the argument of $s$, treated as complex number in $\mathbb{C}_J$. Since the function $q_\alpha$ can also be written as
\begin{equation*}
q_\alpha(s)=(s^2)^{\frac{\alpha}{2}},
\end{equation*}
this function is simply a rewriting of the classical fractional power of the square $s^2$. However, the function $p_\alpha$ is much more natural as a definition of fractional powers of real valued numbers, in the sense that for $\Sc(s)>0$, the angle between $p_\alpha(s)$ and the positive real axis changes by the factor $\alpha$, while for $\Sc(s)<0$, the angle between $p_\alpha(s)$ and the negative real axis changes with the factor $\alpha$.  \medskip

We remark that the theory of fractional powers of complex operators has been developed by several authors. Without claiming completeness, we mention the early works \cite{Balakrishnan1960,Guzman1976,Guzman1978a,Guzman1978b} and also \cite{Kato1960,Komatsu1966,Komatsu1967,Komatsu1969,Watanabe1961,Yosida1960}. The literature is now very wide and it has developed in several directions. \medskip

In order to give meaning to the fractional powers $p_\alpha(T)$ and $q_\alpha(T)$ of any bisectorial Clifford operator $T$, we will use the so called $H^\infty$-functional calculus. The $H^\infty$-functional calculus, see \cite{Haase,HYTONBOOK1,HYTONBOOK2}, extends the Riesz-Dunford functional calculus \cite{Dunford1988,Rudin1987} to sectorial and bisectorial operators. It is established via a three-step procedure: First, the calculus is defined for functions that exhibit decay at zero and infinity, see Definition~\ref{defi_Omega}. In the second step it will be extended to functions which are continuous at zero and at infinity, see Definition~\ref{defi_Extended}. In the third step, via a regularization procedure, the calculus is extended to functions with polynomial growth, see Definition~\ref{defi_Hinfty} and Remark~\ref{rem_Polynomially_growing_functions}. This theory is of fundamental importance for linear operators on Banach spaces, and it is used to study partial differential equations and time-dependent evolution processes. In our previous works \cite{MS24}, there is already a study of bisectorial operators in the Clifford setting and a detailed investigation of the $H^\infty$-functional calculus is provided.  \medskip

Our approach to fractional diffusion problems, based on the $S$-spectrum, extends beyond the Fourier law with constant coefficients. It applies to the general variable-coefficient case \eqref{Eq_Gradient_formal} and generates the associated non-local diffusion operator. \medskip

To illustrate this, let $v:\mathbb{R}^n\times[0,\infty)\rightarrow\mathbb{R}$ denote the temperature and $q$ the heat flow. For $x=(x_1,\dots,x_n)\in\mathbb{R}^n$ and $t>0$, the classical heat equation is deduced from the following two laws:
\begin{align*}
q(x,t)+\nabla_av(x,t)&=0,\qquad\text{(Fourier's law)} \\
\frac{\partial}{\partial t}v(x,t)+\Div q(x,t)&=0.\qquad\text{(Conservation of energy)}
\end{align*}
In our framework, we define a fractional Fourier law using fractional powers $p_\alpha(\nabla_a)$ of the generalized gradient. Subsequently, the fractional Fourier law $q(x,t)=-p_\alpha(\nabla_a)v(x,t)$, is replaced into the energy conservation law. In doing so we can also formulate the associated time-independent spectral problem of the fractional heat equation
\begin{equation}\label{Eq_Fractional_heat_equation}
\lambda v(x)-\Div p_\alpha(\nabla_a)v(x)=f(x),\qquad\lambda\in\mathbb{C}.
\end{equation}
This approach based on the theory on the $S$-spectrum has advantages that include the following.

\begin{enumerate}
\item[i)] It modifies the Fourier law while preserving the conservation of energy. \medskip

\item[ii)] It is applicable to a large class of operators, including gradients with non-constant coefficients. In the specific case of the standard gradient $\nabla$, we recover the fractional heat equation with the fractional Laplacian defined via Fourier transform or the spectral theorem. \medskip

\item[iii)] Keeping the divergence form of the evolution equation allows for an immediate and natural definition of the weak solution to the fractional evolution problem. \medskip
\end{enumerate}

\textit{The content of the paper}: In Section~\ref{sec_Preliminaries}, we present preliminary material regarding the spectral theory on the $S$-spectrum and recall basic facts concerning bisectorial operators and the functional calculi we will use in this article. In Section~\ref{sec_Fractional_powers}, we define the fractional powers of abstract bisectorial operators.
First, in Section~\ref{sec_Positive_fractional_powers} we define positive fractional powers for operators which are not necessarily injective. We prove classical properties like the power rule or the decomposition rule of these exponents.
Then in Section~\ref{sec_Negative_fractional_powers} we also introduce negative powers, but only for injective operators. Starting from Section~\ref{sec_Gradient}, we apply the abstract results to concrete problems that lead to new fractional equations for heat propagation. Specifically, we establish bisectorial estimates for the gradient operator with non-constant coefficients. In particular, we derive regularity results of weak solutions in Theorem~\ref{thm_Strong_solution}, injectivity of the gradient in Proposition~\ref{prop_Gradient_injective}, and the fact that $\nabla_a$ is a bisectorial operator in Theorem~\ref{thm_Gradient_bisectorial}. All these results allow us to define fractional powers of the gradient, and as a central result, we prove that the fractional power acts on real function as a vector valued function, see Theorem~\ref{thm_palpha_vector_operator}.

\section{Preliminaries on Clifford Algebras and Clifford modules}\label{sec_Preliminaries}

In this section we will fix the algebraic and functional analytic setting of this paper. The underlying algebra in this article will be the real \textit{Clifford algebra} $\mathbb{R}_n$ over $n$ \textit{imaginary units} $e_1,\dots,e_n$, which satisfy the relations
\begin{equation*}
e_i^2=-1\qquad\text{and}\qquad e_ie_j=-e_je_i,\qquad i\neq j\in\{1,\dots,n\}.
\end{equation*}
In this context we will denote

\begin{enumerate}
\item[$\circ$] $\mathbb{R}^{n+1}:=\big\{x_0+x_1e_1+\dots+x_ne_n\;\big|\;x_0,x_1,\dots,x_n\in\mathbb{R}\big\}$ the subset of all \textit{paravectors}, \medskip

\item[$\circ$] $\mathbb{S}:=\big\{x_1e_1+\dots+x_ne_n\;\big|\;x_1^2+\dots+x_n^2=1\big\}$ the set of all imaginary units, \medskip

\item[$\circ$] $\mathbb{C}_J:=\big\{x+Jy\;\big|\;x,y\in\mathbb{R}\}$ the complex hyperplane of the imaginary unit $J\in\mathbb{S}$.
\end{enumerate}

In particular, we will consider \textit{Hilbert modules} over $\mathbb{R}_n$, which are real Hilbert spaces $V$ where the linearity of the space and the inner product is compatible with $\mathbb{R}_n$. The space of bounded, everywhere defined operators $T:V\rightarrow V$ will be denoted by $\mathcal{B}(V)$ and the set of closed, unbounded operators by $\mathcal{K}(V)$. An overview over the notation and special objects related to the Clifford algebra $\mathbb{R}_n$ can be found in \cite[Section~2]{MS24}. \medskip

Differently from complex Hilbert spaces, in Cliffordian Hilbert modules the spectrum of an operator $T$ is connected to the bounded invertibility of the operator
\begin{equation}\label{Eq_Qs}
Q_s[T]:=T^2-2s_0T+|s|^2,\qquad\text{with }\dom(Q_s[T]):=\dom(T^2).
\end{equation}
This suggests the following definition of $S$-spectrum.

\begin{defi}[$S$-Spectrum]\label{defi_S_spectrum}
For every $T\in\mathcal{K}(V)$, its \textit{$S$-resolvent set} and \textit{$S$-sepctrum} are defined as
\begin{equation*}
\rho_S(T):=\big\{s\in\mathbb{R}^{n+1}\;\big|\;Q_s[T]^{-1}\in\mathcal{B}(V)\big\}\qquad\text{and}\qquad\sigma_S(T):=\mathbb{R}^{n+1}\setminus\rho_S(T),
\end{equation*}
where $\mathcal{B}(V)$ is the space of bounded everywhere defined operators. Moreover, for every $s\in\rho_S(T)$ we define the \textit{left} and the \textit{right $S$-resolvent operators}
\begin{equation}\label{Eq_SL_SR}
S_L^{-1}(s,T):=Q_s[T]^{-1}\overline{s}-TQ_s[T]^{-1}\qquad\text{and}\qquad S_R^{-1}(s,T):=(\overline{s}-T)Q_s[T]^{-1}.
\end{equation}
\end{defi}

In this article we will consider the important special class of bisectorial operators. In order to introduce them, we define for every $\omega\in(0,\frac{\pi}{2})$ the open \textit{double sector}
\begin{equation*}
D_\omega:=\big\{re^{J\phi}\;\big|\;r>0,\,J\in\mathbb{S},\,\phi\in(-\omega,\omega)\cup(\pi-\omega,\pi+\omega)\big\},
\end{equation*}

\begin{defi}[Bisectorial operator]\label{defi_Bisectorial_operator}
An operator $T\in\mathcal{K}(V)$ is called \textit{bisectorial of angle} $\omega\in(0,\frac{\pi}{2})$, if its $S$-spectrum is contained in the closed double sector
\begin{equation}\label{Eq_Bisectorial_Spectrum}
\sigma_S(T)\subseteq\overline{D_\omega},
\end{equation}
and for every $\varphi\in(\omega,\frac{\pi}{2})$ there exists some $C_\varphi\geq 0$ such that the left $S$-resolvent \eqref{Eq_SL_SR} satisfies
\begin{equation}\label{Eq_SL_estimate}
\Vert S_L^{-1}(s,T)\Vert\leq\frac{C_\varphi}{|s|},\qquad s\in\mathbb{R}^{n+1}\setminus(D_\varphi\cup\{0\}).
\end{equation}
\end{defi}

For these operators, it is possible to define $f(T)$ via the so called $\omega$-functional calculus \cite[Definition~3.5]{MS24} if $f$ is holomorphic and decays at zero and at infinity, and more importantly via the $H^\infty$-functional calculus \cite[Definition~5.3]{MS24} if $f$ is holomorphic and polynomially growing at zero and infinity. Following early explorations in \cite{ACQS2016,CGdiffusion2018}, the calculus was recently established for unbounded bisectorial operators in Clifford modules \cite{MS24}, thanks to the universality property described in \cite{ADVCGKS}. While for left slice holomorphic functions, this calculus is considered already in \cite[Theorem 7.2.6]{FJBOOK}, $f(T)$ for right holomorphic functions was a bigger challenge and only recently in \cite[Definition~4.7]{CMS25} it was fully understood and developed. \medskip

More precisely, if we denote by $\mathcal{SH}_L(D_\theta)$ the space of left slice-hyperholomorphic functions $f:D_\theta\rightarrow\mathbb{R}_n$, see also \cite[Definition~2.1]{MS24}, we will consider the following two functional calculi in this paper.

\begin{defi}[$\omega$-functional calculus]\label{defi_Omega}
Let $T\in\mathcal{K}(V)$ be bisectorial of angle $\omega\in(0,\frac{\pi}{2})$. Then, for every $f\in\mathcal{SH}_L(D_\theta)$, which satisfies the estimate
\begin{equation}\label{Eq_Omega_estimate}
|f(s)|\leq\frac{C_\alpha|s|^\alpha}{1+|s|^{2\alpha}},\qquad s\in D_\theta,
\end{equation}
we define the \textit{$\omega$-functional calculus} \medskip

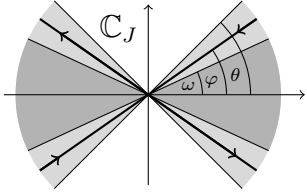
\begin{minipage}{0.4\textwidth}
\begin{center}
\begin{tikzpicture}[scale=0.8]
\fill[black!15] (0,0)--(1.56,1.56) arc (45:-45:2.2)--(0,0)--(-1.56,-1.56) arc (225:135:2.2);
\fill[black!30] (0,0)--(2,0.93) arc (25:-25:2.2)--(0,0)--(-2,-0.93) arc (205:155:2.2);
\draw (2,0.93)--(-2,-0.93);
\draw (2,-0.93)--(-2,0.93);
\draw (1.56,1.56)--(-1.56,-1.56);
\draw (1.56,-1.56)--(-1.56,1.56);
\draw (0.9,0) arc (0:25:0.9) (0.67,-0.1) node[anchor=south] {\tiny{$\omega$}};
\draw (1.3,0) arc (0:35:1.3) (1.05,-0.05) node[anchor=south] {\tiny{$\varphi$}};
\draw (1.7,0) arc (0:45:1.7) (1.45,0.05) node[anchor=south] {\tiny{$\theta$}};
\draw[thick] (1.8,1.26)--(-1.8,-1.26);
\draw[thick] (1.8,-1.26)--(-1.8,1.26);
\draw[thick,->] (1.8,1.26)--(1.47,1.03);
\draw[thick,->] (0,0)--(1.47,-1.03);
\draw[thick,->] (-1.8,-1.26)--(-1.47,-1.03);
\draw[thick,->] (0,0)--(-1.47,1.03);
\draw[->] (-2.4,0)--(2.6,0);
\draw[->] (0,-1.5)--(0,1.5) node[anchor=north east] {\large{$\mathbb{C}_J$}};
\end{tikzpicture}
\end{center}
\end{minipage}
\begin{minipage}{0.59\textwidth}
\begin{equation}\label{Eq_Omega}
f(T):=\frac{1}{2\pi}\int_{\partial D_\varphi\cap\mathbb{C}_J}S_L^{-1}(s,T)ds_Jf(s),
\end{equation}
where $\varphi\in(\omega,\theta)$ and $J\in\mathbb{S}$ are arbitrary, and \eqref{Eq_Omega} is independent of their choice.
\end{minipage}

\medskip We will denote the space of all $f\in\mathcal{SH}_L(D_\theta)$ which satisfy \eqref{Eq_Omega_estimate} by $\mathcal{SH}_L^0(D_\theta)$. If $f$ is also intrinsic, we will call the space $\mathcal{N}^0(D_\theta)$.
\end{defi}

\begin{defi}[Extended $\omega$-functional calculus]\label{defi_Extended}
Let $T\in\mathcal{K}(V)$ be bisectorial of angle $\omega\in(0,\frac{\pi}{2})$. Then, for every $f\in\mathcal{SH}_L(D_\theta)$, $\theta\in(\omega,\frac{\pi}{2})$, which admits the decomposition
\begin{equation}\label{Eq_Extended_decomposition}
f(s)=f_\infty+\frac{1}{1+s^2}(f_0-f_\infty)+\widetilde{f}(s),\qquad s\in D_\theta,
\end{equation}
for some $f_0,f_\infty\in\mathbb{R}_n$ and $\widetilde{f}\in\mathcal{SH}_L^0(D_\theta)$, we define the \textit{extended $\omega$-functional calculus}
\begin{equation}\label{Eq_Extended}
f(T):=f_\infty+(1+T^2)^{-1}(f_0-f_\infty)+\widetilde{f}(T).
\end{equation}
The space of functions $f\in\mathcal{SH}_L(D_\theta)$ which satisfy the decomposition \eqref{Eq_Extended_decomposition} is denoted by $\mathcal{SH}_L^\bnd(D_\theta)$. If $f$ is also intrinsic, we call the space $\mathcal{N}^\bnd(D_\theta)$.
\end{defi}

\begin{defi}[$H^\infty$-functional calculus]\label{defi_Hinfty}
Let $T\in\mathcal{K}(V)$ be bisectorial of angle $\omega\in(0,\frac{\pi}{2})$. Then for every $f\in\mathcal{SH}_L(D_\theta)$, $\theta\in(\omega,\frac{\pi}{2})$, for which there exists a function $e\in\mathcal{N}^\bnd(D_\theta)$ with the properties
\begin{equation*}
ef\in\mathcal{N}^\bnd(D_\theta)\qquad\text{and}\qquad e(T)\text{ is injective},
\end{equation*}
we define the \textit{$H^\infty$-functional calculus}
\begin{equation}\label{Eq_Hinfty}
f(T):=e(T)^{-1}(ef)(T).
\end{equation}
This definition is independent of the choice of the regularizer function $e$.
\end{defi}

\begin{rem}\label{rem_Polynomially_growing_functions}
Let $T\in\mathcal{K}(V)$ be bisectorial of angle $\omega\in(0,\frac{\pi}{2})$, and $f\in\mathcal{SH}_L(D_\theta)$, for some $\theta\in(\omega,\frac{\pi}{2})$. \medskip

\begin{enumerate}
\item[i)] If there exists some $f_0\in\mathbb{R}_n$ such that
\begin{equation*}
|f(s)-f_0|\leq C\max\{|s|^\gamma,|s|^\delta\},\qquad s\in D_\theta,
\end{equation*}
for some $0<\gamma\leq\delta$ and $C\geq 0$, then one can apply the $H^\infty$-functional calculus with the regularizer function $e(s)=\frac{1}{(1+s^2)^n}$, for some $n>\frac{\delta}{2}$.

\item[ii)] If the operator $T$ is injective, we can also allow the function $f$ to be polynomially growing at $0$ and at $\infty$, i.e.,
\begin{equation*}
|f(s)|\leq C\Big(\frac{1}{|s|^\gamma}+|s|^\gamma\Big),\qquad s\in D_\theta,
\end{equation*}
for some $\gamma>0$ and $C\geq 0$. Then we can apply the $H^\infty$-functional calculus with the regularizer function $e(s)=\frac{s^n}{(1+s^2)^n}$, for some $n>\gamma$.
\end{enumerate}

\end{rem}

\section{Fractional powers of bi-sectorial operators}\label{sec_Fractional_powers}

In this section, we define and analyse two versions of fractional powers of bisectorial operators $T$ with real  exponent.
More precisely, with the functions in \eqref{Eq_palpha_qalpha}, we investigate the operators $p_\alpha(T)$ and $q_\alpha(T)$ for every $\alpha\in\mathbb{R}$. Since the logarithm $\ln(s)$ in \eqref{Eq_ln} and hence the fractional power $s^\alpha$ in \eqref{Eq_salpha} are intrinsic functions for $\Sc(s)>0$, both functions
\begin{equation*}
p_\alpha(s)\text{ and }q_\alpha(s)\text{ are intrinsic for }\Sc(s)\neq 0.
\end{equation*}
Moreover, the absolute values of $p_\alpha$ and $q_\alpha$ are given by
\begin{equation*}
|p_\alpha(s)|=|q_\alpha(s)|=|s|^\alpha,\qquad\Sc(s)\neq 0.
\end{equation*}
This shows that the functions $p_\alpha,q_\alpha$ are polynomially growing, and it will be possible to give meaning to $p_\alpha(T)$ and $q_\alpha(T)$ via the $H^\infty$-functional calculus, Definition~\ref{defi_Hinfty}. \medskip

More precisely, for $\alpha>0$ the functions $p_\alpha$ and $q_\alpha$ are vanishing as
$s\rightarrow 0$ and are polynomially growing as $s\rightarrow\infty$. In this case, $p_\alpha(T)$ and $q_\alpha(T)$ can be defined for every bisectorial operator $T$, see Definition~\ref{defi_palphaT_positive}. Conversely, for $\alpha\leq 0$, both functions $p_\alpha$ and $q_\alpha$ are not vanishing as $s\rightarrow 0$, they are bounded for $\alpha=0$ or even polynomially growing for $\alpha<0$. Hence, we need to regularize the function at $s=0$, which makes it necessary to assume $T$ injective, see Definition~\ref{defi_palphaT_negative}. \medskip

For this reason, we will split the considerations of positive and negative powers into the two distinct Subsection~\ref{sec_Positive_fractional_powers} and Subsection~\ref{sec_Negative_fractional_powers}.

\subsection{Positive fractional powers}\label{sec_Positive_fractional_powers}

In this subsection we define the operators $p_\alpha(T)$ and $q_\alpha(T)$ for positive powers $\alpha>0$, for bisectorial operators $T\in\mathcal{K}(V)$, without the need of $T$ being injective. As central properties, we will find the product rule in Theorem~\ref{thm_Product_rule_positive_powers} as well as the composition rule of fractional powers in Theorem~\ref{thm_Composition_positive_powers}.

\begin{defi}\label{defi_palphaT_positive}
Let $T\in\mathcal{K}(V)$ be bisectorial. Then for every $\alpha>0$, we define
\begin{equation*}
p_\alpha(T)\text{ and }q_\alpha(T)\text{ via the }H^\infty\text{-functional calculus \eqref{Eq_Hinfty}}.
\end{equation*}
A regularizer, is for every integer $n>\frac{\alpha}{2}$, is given by
\begin{equation}\label{Eq_Regularizer_positive}
e(s)=\frac{1}{(1+s^2)^n}.
\end{equation}
\end{defi}

\begin{rem}
Note that the 'trivial' case $\alpha=0$ is not covered by Definition~\ref{defi_palphaT_positive}. Since the functions $p_0(s)$ and $q_0(s)$ do not decay at zero but are only bounded, the definition of $p_0(T)$ and $q_0(T)$ is subtle. \medskip

On the one hand, the function $q_0(s)=1$ in \eqref{Eq_palpha_qalpha} is the constant function. This means, also the operator
\begin{equation*}
q_0(T)=1
\end{equation*}
can be realized as the identity operator, with the help of the extended $\omega$-functional calculus \eqref{Eq_Extended}. On the other hand, the function $p_0(s)=\sgn(s)$ in \eqref{Eq_palpha_qalpha} reduces to the sign function (understood as the sign function of the scalar part of $s$). Hence, it is constant on both sectors, but the respective limit $s\rightarrow 0$ and $s\rightarrow\infty$ depend from which direction one approaches $0$ and $\infty$. Hence, the extended $\omega$-functional calculus is not applicable for $p_0$, and we need to use the full $H^\infty$-functional calculus and regularize the function at $0$ as well as at $\infty$. This in particular means that we need to assume that $T$ is injective in order to define
\begin{equation*}
p_0(T)=\sgn(T),
\end{equation*}
see Definition~\ref{defi_palphaT_negative}. We also a priori do not know if the resulting operator is bounded. Only if we additionally assume that the operator $T$ satsify quadratic estimates, see \cite{CMS26}, we know that $p_0(T)$ is a bounded operator.
\end{rem}

\begin{rem}\label{rem_Integer_powers}
Consider the special case $\alpha=n$ being an odd positive integer, or $\alpha=m$ being an even positive integer. Then there is $p_n(s)=s^n$ and $q_m(s)=s^m$, which by \cite[Theorem~5.9]{MS24} means
\begin{equation*}
p_n(T)=T^n,\qquad\text{and}\qquad q_m(T)=T^m,
\end{equation*}
in the sense that action as well as domains of both sides of the equation coincide.
\end{rem}

In order to find the first central property of fractional powers, the product rule in Theorem~\ref{thm_Product_rule_positive_powers}, we have to improve the well known product rule \cite[Theorem~5.7]{MS24} of the $H^\infty$-functional calculus, which in general only gives an operator inclusion. It is already known that if for example the second operator $f(T)$ is bounded, it is satisfied with equality. However, for both operators $f(T)$ and $g(T)$ being unbounded,
there is no existing result which ensures equality. Roughly speaking, the following Proposition~\ref{prop_Product_rule_positive} shows that for functions $f,g$ which, up to some constant, are both vanishing at $0$ and are polynomially bounded at $\infty$, and for which $1/g$ vanishes at $\infty$ and polynomially grows at $0$, the product rule holds with equality.

\begin{prop}\label{prop_Product_rule_positive}
Let $T\in\mathcal{K}(V)$ be bisectorial of angle $\omega\in(0,\frac{\pi}{2})$, and $f\in\mathcal{SH}_L(D_\theta)$, $g\in\mathcal{N}(D_\theta)$, for some $\theta\in(\omega,\frac{\pi}{2})$. If there exist $0<\gamma\leq\delta$, as well as constants $0<C_1\leq C_2$, such that for some $f_0\in\mathbb{R}_n$, $g_0\in\mathbb{R}$ there hold the estimates
\begin{subequations}
\begin{align}
C_1\min\{|s|^\delta,|s|^\gamma\}\leq|g(s)-g_0|\leq C_2\max\{|s|^\delta,|s|^\gamma\},\qquad s\in D_\theta, \label{Eq_Product_rule_positive_g} \\
|f(s)-f_0|\leq C_2\max\{|s|^\delta,|s|^\gamma\},\qquad s\in D_\theta. \label{Eq_Product_rule_positive_f}
\end{align}
\end{subequations}
Then, we have the product rule
\begin{equation*}
(gf)(T)=g(T)f(T).
\end{equation*}
\end{prop}

\begin{proof}
From the product rule \cite[Theorem~5.7]{MS24} of the $H^\infty$-functional calculus, we already know the operator inclusion
\begin{equation}\label{Eq_Product_rule_positive_3}
(gf)(T)\supseteq g(T)f(T),
\end{equation}
with the domains connected by
\begin{equation*}
\dom((gf)(T))\cap\dom(f(T))=\dom(g(T)f(T)).
\end{equation*}
This means, in order to show equality in \eqref{Eq_Product_rule_positive_3}, we have to verify
\begin{equation}\label{Eq_Product_rule_positive_2}
\dom((gf)(T))\subseteq\dom(f(T)).
\end{equation}
To do so, let $v\in\dom((gf)(T))$. Due to the upper bounds \eqref{Eq_Product_rule_positive_g} and \eqref{Eq_Product_rule_positive_f}, it is possible to consider the regularizer function $e$ in \eqref{Eq_Regularizer_positive}, with $n>\frac{\delta}{2}$, for both functions $g$ and $f$, see also Remark~\ref{rem_Polynomially_growing_functions}. Then $e^2$ is a regularizer of the product $gf$ and by the definition \eqref{Eq_Hinfty} of the $H^\infty$-functional calculus, the condition $v\in\dom((gf)(T))$ translates into
\begin{equation*}
(e^2gf)(T)v\in\dom(e^2(T)^{-1})=\ran(e^2(T))=\ran\big((1+T^2)^{-2n}\big)=\dom(T^{4n}).
\end{equation*}
Next, due to the lower bound \eqref{Eq_Product_rule_positive_g}, and by the choice $n>\frac{\delta}{2}$, there is
$$\frac{s^{2n}e}{g}\in\mathcal{N}^0(D_\theta),
$$ a function for which we can apply the $\omega$-functional calculus. From the operator inclusion in \cite[Corollary~3.18~ii)]{MS24}, it then follows that also
\begin{equation*}
\Big(\frac{s^{2n}e}{g}\Big)(T)(e^2gf)(T)v\in\dom(T^{4n}).
\end{equation*}
However, with the product rule \cite[Theorem~4.7]{MS24} of the extended $\omega$-functional calculus, we can rewrite this expression as
\begin{equation*}
\Big(\frac{s^{2n}e}{g}\Big)(T)(e^2gf)(T)v=(s^{2n}e^3f)(T)v=(s^{2n}e^2)(T)(ef)(T)v.
\end{equation*}
Altogether, this means that
$$
(s^{2n}e^2)(T)(ef)(T)v\in\dom(T^{4n}),
$$
 and because we can explicitly write
 $$
 (s^{2n}e^2)(T)=T^{2n}(1+T^2)^{-2n},
 $$
  this implies
\begin{equation*}
(ef)(T)v\in\dom(T^{2n})=\ran(e(T))=\dom(e(T)^{-1}).
\end{equation*}
By the definition \eqref{Eq_Hinfty} of the $H^\infty$-functional calculus, this proves $v\in\dom(f(T))$, and we have verified the domain inclusion \eqref{Eq_Product_rule_positive_2}.
\end{proof}

The product rule of Proposition~\ref{prop_Product_rule_positive} can now be applied to the functions $p_\alpha$ and $q_\alpha$, to obtain the following product rule of fractional powers.

\begin{thm}\label{thm_Product_rule_positive_powers}
Let $T\in\mathcal{K}(V)$ be bisectorial. Then for $\alpha,\beta>0$ we have the power rules

\begin{enumerate}
\item[i)] $p_{\alpha+\beta}(T)=p_\alpha(T)q_\beta(T)=q_\alpha(T)p_\beta(T)$; \medskip

\item[ii)] $q_{\alpha+\beta}(T)=p_\alpha(T)p_\beta(T)=q_\alpha(T)q_\beta(T)$.
\end{enumerate}
\end{thm}

\begin{proof}
It is immediate that for every $\Sc(s)\neq 0$, there is
\begin{equation}\label{Eq_Product_rule_positive_powers_1}
\begin{split}
p_{\alpha+\beta}(s)=&p_\alpha(s)q_\beta(s)=q_\alpha(s)p_\beta(s), \\
q_{\alpha+\beta}(s)=&p_\alpha(s)p_\beta(s)=q_\alpha(s)q_\beta(s).
\end{split}
\end{equation}
The stated product rules are then a consequence of Proposition~\ref{prop_Product_rule_positive}.
\end{proof}

This product rule of fractional powers can now also be used to get a domain inclusion in between different order fractional powers.

\begin{cor}\label{cor_Operator_domain_positive_powers}
Let $T\in\mathcal{K}(V)$ be bisectorial. Then for every $\beta>\alpha>0$, we have

\begin{enumerate}
\item[i)] $\dom(p_\beta(T))\subseteq\dom(p_\alpha(T))\cap\dom(q_\alpha(T))$; \medskip

\item[ii)] $\dom(q_\beta(T))\subseteq\dom(p_\alpha(T))\cap\dom(q_\alpha(T))$.
\end{enumerate}
\end{cor}

\begin{rem}
Note that in the domain inclusions of Corollary~\ref{cor_Operator_domain_positive_powers} we really need $\beta>\alpha$. In the case $\beta=\alpha$, we in general have neither inclusion
\begin{equation*}
\dom(p_\alpha(T))\substack{\not\subset \\ \not\supset}\dom(q_\alpha(T)).
\end{equation*}
\end{rem}

\begin{thm}\label{thm_palpha_bisectorial_positive}
Let $T\in\mathcal{K}(V)$ be bisectorial of angle $\omega\in(0,\frac{\pi}{2})$.

\begin{enumerate}
\item[i)] If $0<\alpha<\frac{\pi}{2\omega}$, then $p_\alpha(T)$ is bisectorial of angle $\alpha\omega$. \medskip

\item[ii)] If $0<\alpha<\frac{\pi}{\omega}$, then $q_\alpha(T)$ is sectorial of angle $\alpha\omega$.
\end{enumerate}
\end{thm}

\begin{proof}
i)\;\;Let us fix $\varphi\in(\alpha\omega,\frac{\pi}{2})$. Then, for every $s\in\mathbb{R}^{n+1}\setminus(D_\varphi\cup\{0\})$, we are able to define
\begin{equation}\label{Eq_palpha_bisectorial_6}
h_s(\xi):=\frac{1}{p_\alpha(\xi)^2-2s_0p_\alpha(\xi)+|s|^2},\qquad\xi\in\overline{D_{\varphi'}},
\end{equation}
for every $\omega<\varphi'<\max\{\frac{\varphi}{\alpha},\frac{\pi}{2}\}$. Note that the denominator never vanishes, because
\begin{equation*}
\arg(p_\alpha(\xi))=\begin{cases} \alpha\arg(\xi), & \Sc(\xi)>0, \\ \alpha(\arg(\xi)-\pi)+\pi, & \Sc(\xi)<0, \end{cases}\qquad\xi\in\overline{D_{\varphi'}},
\end{equation*}
and hence $p_\alpha(\xi)\in\overline{D_{\alpha\varphi'}}\subseteq D_\varphi\cup\{0\}$, while $s\in\mathbb{R}^{n+1}\setminus(D_\varphi\cup\{0\})$. The convention of $\arg$, which is used here, is $\arg(\xi)\in(-\frac{\pi}{2},\frac{\pi}{2})\cup(\frac{\pi}{2},\frac{3\pi}{2})$, for every $\Sc(\xi)\neq 0$. Since $|p_\alpha(\xi)|=|\xi|^\alpha$, it is now clear that $h_s\in\mathcal{N}^\bnd(D_{\varphi'})$, with the values $h_{s,0}=\frac{1}{|s|^2}$ and $h_{s,\infty}=0$ in Definition~\ref{defi_Extended}. Hence, $h_s(T)$ is a well defined, bounded operator via the extended $\omega$-functional calculus of Definition~\ref{defi_Extended}. In order to show that $s\in\rho_S(p_\alpha(T))$, we have to show that the operator $$
Q_s[p_\alpha(T)]=p_\alpha(T)^2-2s_0p_\alpha(T)+|s|^2
$$
is boundedly invertible. By the product rule in Proposition~\ref{prop_Product_rule_positive}, we can write the polynomial functional calculus of $Q_s[p_\alpha(T)]$ as the $H^\infty$-functional calculus
\begin{equation*}
Q_s[p_\alpha(T)]=p_\alpha(T)(p_\alpha-2s_0)(T)+|s|^2=(p_\alpha(p_\alpha-2s_0))(T)+|s|^2=(p_\alpha^2-2s_0p_\alpha+|s|^2)(T).
\end{equation*}
Consequently, by the product rule \cite[Theorem~5.7]{MS24} of the $H^\infty$-functional calculus, we obtain
\begin{align*}
h_s(T)Q_s[p_\alpha(T)]&\subseteq\big(h_s(p_\alpha^2-2s_0p_\alpha+|s|^2)\big)(T)=1,\qquad\text{and also} \\
Q_s[p_\alpha(T)]h_s(T)&=\big((p_\alpha^2-2s_0p_\alpha+|s|^2)h_s\big)(T)=1.
\end{align*}
This shows that $Q_s[p_\alpha(T)]$ is bijective, with $Q_s[p_\alpha(T)]^{-1}=h_s(T)$ being a bounded operator. Hence we have proven $s\in\rho_S(p_\alpha(T))$. Since $s\in\mathbb{R}^{n+1}\setminus(D_\varphi\cup\{0\})$ was arbitrary, this means that $\sigma_S(p_\alpha(T))\subseteq D_\varphi\cup\{0\}$ for every $\varphi\in(\alpha\omega,\frac{\pi}{2})$. Taking the union over all $\varphi$, then gives
\begin{equation*}
\sigma_S(p_\alpha(T))\subseteq\overline{D_{\alpha\omega}}.
\end{equation*}
In order to get the estimate \eqref{Eq_SL_estimate} of the left $S$-resolvent operator $S_L^{-1}(s,p_\alpha(T))$, we first rewrite it as the extended $\omega$-functional calculus
\begin{equation}\label{Eq_palpha_bisectorial_5}
S_L^{-1}(s,p_\alpha(T))=h_s(T)\overline{s}-p_\alpha(T)h_s(T)=(h_s\overline{s}-p_\alpha h_s)(T),
\end{equation}
using the product rule \cite[Theorem~5.7]{MS24}, which is satisfied with equality because the operator $h_s(T)$ is bounded.
Since it is clear that $h_s\overline{s}-p_\alpha h_s\in\mathcal{N}^\bnd(D_{\varphi'})$ from Definition~\ref{defi_Extended}, with limit $\frac{1}{s}$ as $\Xi\rightarrow 0$ and limit $0$ when $\xi\rightarrow\infty$, the extended $\omega$-functional calculus \eqref{Eq_Extended} of $S_L^{-1}(s,p_\alpha(T))$ can be written as
\begin{equation}\label{Eq_palpha_bisectorial_3}
S_L^{-1}(s,p_\alpha(T))=(|s|^{\frac{2}{\alpha}}+T^2)^{-1}\frac{|s|^{\frac{2}{\alpha}}}{s}+\Big(\underbrace{h_s(\xi)\overline{s}-p_\alpha(\xi)h_s(\xi)-\frac{1}{|s|^{\frac{2}{\alpha}}+\xi^2}\frac{|s|^{\frac{2}{\alpha}}}{s}}_{=:g_s(\xi)}\Big)(T).
\end{equation}
Hence, the remaining task is to estimate the operator $g_s(T)$, defined via the $\omega$-functional calculus \eqref{Eq_Omega}. First, for every $\xi\in\overline{D_{\varphi'}} \cap \mathbb{C}_J$, where $J \in \mathbb{S}$ is an imaginary unit such that $s \in \mathbb{C}_J$, we can estimate
\begin{align}
|g_s(\xi)|&=\bigg|\frac{1}{s-p_\alpha(\xi)}-\frac{|s|^{\frac{2}{\alpha}}}{(|s|^{\frac{2}{\alpha}}+\xi^2)s}\bigg|\leq\frac{|s||\xi|^2+|s|^{\frac{2}{\alpha}}|\xi|^\alpha}{|s||s-p_\alpha(\xi)|||s|^{\frac{2}{\alpha}}+\xi^2|} \notag \\
&\leq\frac{4}{|e^{J\varphi}-e^{J\alpha\varphi'}||1+e^{2J\varphi'}|}\frac{|s||\xi|^2+|s|^{\frac{2}{\alpha}}|\xi|^\alpha}{|s|(|s|+|\xi|^\alpha)(|s|^{\frac{2}{\alpha}}+|\xi|^2)}, \label{Eq_palpha_bisectorial_4}
\end{align}
where in the first equality we used that $p_\alpha(\xi)\in\mathbb{C}_J$ is in the same complex plane as $s$, and hence the term $h_s(\xi)\overline{s}-p_\alpha(\xi)h_s(\xi)$ simplifies to $(s-p_\alpha(\xi))^{-1}$. Moreover, in the second line of \eqref{Eq_palpha_bisectorial_4}, we used
\begin{align}
\frac{|s-p_\alpha(\xi)|}{|s|+|\xi|^\alpha}&\geq\inf\limits_{\phi\in[\varphi,\frac{\pi}{2}]}\inf\limits_{\phi'\in[0,\varphi']}\frac{\big||s|e^{J\phi}-|\xi|^\alpha p_\alpha(e^{J\phi'})|}{|s|+|\xi|^\alpha} \notag \\
&\geq\inf\limits_{\phi\in[\varphi,\frac{\pi}{2}]}\inf\limits_{\phi'\in[0,\varphi']}\frac{\big|e^{J\phi}-e^{J\alpha\phi'}\big|}{2}=\frac{|e^{J\varphi}-e^{J\alpha\varphi'}|}{2}, \label{Eq_palpha_bisectorial_7}
\end{align}
where in the first inequality, the difference $|s-p_\alpha(\xi)|$ becomes minimal if $s$ and $\xi$ are in the same quadrant, and by symmetry we chose the first one. In the second inequality we used that the minimum with respect to $|s|>0$ and $|\xi|>0$ is attained for $|s|=|\xi|^\alpha$. In the second line of \eqref{Eq_palpha_bisectorial_4} we also used the similar inequality
\begin{equation}\label{Eq_palpha_bisectorial_8}
\frac{||s|^{\frac{2}{\alpha}}+\xi^2|}{|s|^{\frac{2}{\alpha}}+|\xi|^2}\geq\inf\limits_{\phi'\in[-\varphi',\varphi']}\frac{||s|^{\frac{2}{\alpha}}+|\xi|^2e^{2J\phi'}|}{|s|^{\frac{2}{\alpha}}+|\xi|^2}\geq\inf\limits_{\phi'\in[-\varphi',\varphi']}\frac{|1+e^{2J\phi'}|}{2}=\frac{|1+e^{2J\varphi'}|}{2}.
\end{equation}
Now that we have estimated the function $g_s$ in \eqref{Eq_palpha_bisectorial_4}, we can combine it with the bound \eqref{Eq_SL_estimate} of the $S$-resolvent of $T$, to obtain for every $s\in\mathbb{R}^{n+1}\setminus(D_\varphi\cup\{0\})$ a bound of the operator norm of $g_s(T)$. To do so, we estimate the $\omega$-functional calculus by
\begin{align}
\Vert g_s(T)\Vert&=\frac{1}{2\pi}\bigg\Vert\int_{\partial D_{\varphi'}\cap\mathbb{C}_J}S_L^{-1}(\xi,T)d\xi_Jg_s(\xi)\bigg\Vert \notag \\
&\leq\frac{8C_{\varphi'}}{\pi|e^{J\varphi}-e^{J\alpha\varphi'}||1+e^{2J\varphi'}|}\int_0^\infty\frac{|s|r^2+|s|^{\frac{2}{\alpha}}r^\alpha}{r|s|(|s|+r^\alpha)(|s|^{\frac{2}{\alpha}}+r^2)}dr \notag \\
&=\frac{8C_{\varphi'}}{\pi|e^{J\varphi}-e^{J\alpha\varphi'}||1+e^{2J\varphi'}||s|}\int_0^\infty\frac{r^2+r^\alpha}{r(1+r^\alpha)(1+r^2)}dr \notag \\
&=\frac{8(1+\frac{2}{\alpha})C_{\varphi'}}{\pi|e^{J\varphi}-e^{J\alpha\varphi'}||1+e^{2J\varphi'}||s|}, \label{Eq_palpha_bisectorial_1}
\end{align}
where in the third line we substituted $r\rightarrow|s|^{\frac{1}{\alpha}}r$, and in the last we used $1+r^\alpha\geq\max\{1,r^\alpha\}$, as well as $1+r^2\geq\max\{1,r^2\}$, in order to compute the integral explicitly. Using now the bound \eqref{Eq_SL_estimate} together with \cite[Lemma 3.2 ii)]{CMS25} for the first, and \eqref{Eq_palpha_bisectorial_1} for the second term in in \eqref{Eq_palpha_bisectorial_3}, we get
\begin{align*}
\Vert S_L^{-1}(s,p_\alpha(T))\Vert&\leq|s|^{\frac{2}{\alpha}-1}\Vert Q_{|s|^{\frac{1}{\alpha}}J}[T]^{-1}\Vert+\Vert g_s(T)\Vert \\
&\leq\bigg(2C_\varphi^2+\frac{8(1+\frac{2}{\alpha})C_{\varphi'}}{\pi|e^{J\varphi}-e^{J\alpha\varphi'}||1+e^{2J\varphi'}|}\bigg)\frac{1}{|s|}.
\end{align*}
ii)\;\;In order to show that $q_\alpha(T)$ is sectorial of angle $\alpha\omega$, one can basically do the same proof as for $p_\alpha(T)$. Instead of the function $h_s$ in \eqref{Eq_palpha_bisectorial_6}, we consider for some fixed $\varphi\in(\alpha\omega,\pi)$ and $s\in\mathbb{R}^{n+1}\setminus(S_\varphi\cup\{0\})$, the function
\begin{equation*}
h_s(\xi):=\frac{1}{q_\alpha(\xi)^2-2s_0q_\alpha(\xi)+|s|^2},\qquad\xi\in\overline{S_{\varphi'}},
\end{equation*}
for every $0<\varphi'<\max\{\frac{\varphi}{\alpha},\pi\}$. This function is well defined because
\begin{equation*}
\arg(q_\alpha(\xi))=\begin{cases} \alpha\arg(\xi), & \Sc(\xi)>0, \\ \alpha(\arg(\xi)-\pi), & \Sc(\xi)<0, \end{cases} \qquad\xi\in\overline{S_{\varphi'}},
\end{equation*}
and hence $q_\alpha(\xi)\in\overline{S_{\alpha\varphi'}}\subseteq S_\varphi$, while $s\in\mathbb{R}^{n+1}\setminus S_\varphi$.
\end{proof}

\begin{prop}
Let $T\in\mathcal{K}(V)$ be bisectorial. Then for every $\alpha\in(0,1]$, there is
\begin{equation*}
\ker(p_\alpha(T))=\ker(q_\alpha(T))=\ker(T).
\end{equation*}
\end{prop}

\begin{proof}
We will only consider $p_\alpha$, the proof for $q_\alpha$ is similar. The inclusion $\ker(T)\subseteq\ker(p_\alpha(T))$ is already proved in \cite[Proposition~5.6]{MS24}. For the inverse inclusion let $v\in\ker(p_\alpha(T))$. Since $\alpha\in(0,1]$, there is $\frac{sp_{-\alpha}(s)}{1+s^2}\in\mathcal{N}^\bnd(D_\theta)$, for any $\theta\in(0,\frac{\pi}{2})$. This means, the operator $(\frac{sp_{-\alpha}(s)}{1+s^2})(T)$ is defined via the extended $\omega$-functional calculus \eqref{Eq_Extended}, and with the product rule \cite[Theorem~5.7]{MS24} of the $H^\infty$-functional calculus, we get
\begin{equation*}
0=\Big(\frac{sp_{-\alpha}(s)}{1+s^2}\Big)(T)p_\alpha(T)v=\Big(\frac{sp_{-\alpha}(s)p_\alpha(s)}{1+s^2}\Big)(T)v=\Big(\frac{s}{1+s^2}\Big)(T)v=T(1+T^2)^{-1}v.
\end{equation*}
Applying now the operator $(1+T^2)$ onto this equation, gives $Tv=0$.
\end{proof}

\begin{prop}\label{prop_Composition_rule_not_injective}
Let $T\in\mathcal{K}(V)$ be bisectorial of angle $\omega\in(0,\frac{\pi}{2})$, and $0<\alpha<\frac{\pi}{2\omega}$. Consider a function $f\in\mathcal{SH}_L(D_\theta)$, with $\theta\in(\alpha\omega,\frac{\pi}{2})$, which for some $f_0\in\mathbb{R}_n$ satisfies the estimate
\begin{equation}\label{Eq_Composition_rule_not_injective}
|f(s)-f_0|\leq C\max\{|s|^\gamma,|s|^\delta\},\qquad s\in D_\theta,
\end{equation}
for some $0<\gamma\leq\delta$ and $C\geq 0$.  Then  $f\circ p_\alpha\in\mathcal{SH}_L(D_{\theta'})$, for every $\theta'<\min\{\frac{\theta}{\alpha},\frac{\pi}{2}\}$ and there holds \medskip

\begin{enumerate}
\item[i)] $f(p_\alpha(T))=(f\circ p_\alpha)(T)$, \medskip

\item[ii)] $f(q_\alpha(T))=(f\circ q_\alpha)(T)$.
\end{enumerate}
\end{prop}

\begin{proof}
We will only proof  i), since ii) follows similarly. We split the proof in three steps, first for the $\omega$-functional calculus \eqref{Eq_Omega}, then for the extended $\omega$-functional calculus \eqref{Eq_Extended} and finally for the $H^\infty$-functional calculus \eqref{Eq_Hinfty}. \medskip

In the \textit{first step} let us assume that $f\in\mathcal{SH}_L^0(D_\theta)$. Since we already know that $p_\alpha(T)$ is bisectorial of angle $\alpha\omega$, we can write for any $\varphi\in(\alpha\omega,\frac{\pi}{2})$ and $J\in\mathbb{S}$ the $\omega$-functional calculus as
\begin{equation}\label{Eq_Composition_positive_1}
f(p_\alpha(T))=\frac{1}{2\pi}\int_{\partial D_\varphi\cap\mathbb{C}_J}S_L^{-1}(s,p_\alpha(T))ds_Jf(s).
\end{equation}
Next, for every fixed $s\in\partial D_\varphi\cap\mathbb{C}_J$, we can write the left $S$-resolvent, similar as in \eqref{Eq_palpha_bisectorial_3}, as
\begin{equation}\label{Eq_Composition_positive_4}
S_L^{-1}(s,p_\alpha(T))=(1+T^2)^{-1}\frac{1}{s}+\Big(\underbrace{h_s(\xi)\overline{s}-p_\alpha(\xi)h_s(\xi)-\frac{1}{1+\xi^2}\frac{1}{s}}_{=:g_s(\xi)}\Big)(T).
\end{equation}
Note that we have
\begin{equation}\label{Eq_Composition_positive_3}
\int_{\partial D_\varphi\cap\mathbb{C}_J}\frac{1}{s}ds_Jf(s)=\lim\limits_{\varepsilon\rightarrow 0^+}\lim\limits_{R\rightarrow\infty}\int_{\partial(D_\varphi\cap B_\varepsilon(0)^c\cap B_R(0))\cap\mathbb{C}_J}\frac{1}{s}ds_Jf(s)=0,
\end{equation}
because the integrand $f/ s$ has no singularity inside the domain $D_\varphi\cap B_\varepsilon(0)^c\cap B_R(0)\cap\mathbb{C}_J$. The limit exists because of the decay of the function $f$ at zero and at infinity. Hence, the integral \eqref{Eq_Composition_positive_1} reduces to
\begin{align}
f(p_\alpha(T))&=\frac{1}{2\pi}\int_{\partial D_\varphi\cap\mathbb{C}_J}g_s(T)ds_Jf(s) \notag \\
&=\frac{1}{4\pi^2}\int_{\partial D_\varphi\cap\mathbb{C}_J}\bigg(\int_{\partial D_{\varphi'}\cap\mathbb{C}_J}S_L^{-1}(\xi,T)d\xi_Jg_s(\xi)\bigg)ds_Jf(s), \label{Eq_Composition_positive_2}
\end{align}
where in the second line we wrote the $\omega$-functional calculus of $g_s(T)$ with some freely chosen $\omega<\varphi'<\max\{\frac{\varphi}{\alpha},\frac{\pi}{2}\}$. Note that $\varphi'<\frac{\varphi}{\alpha}$ is needed for $g_s(\xi)$ in \eqref{Eq_Composition_positive_4} to be well defined. Using the estimate \eqref{Eq_SL_estimate} of the $S$-resolvent, a similar estimate of $h_s$ as in  \eqref{Eq_palpha_bisectorial_4}, as well as
$$
|f(s)|\leq\frac{C_\beta|s|^\beta}{1+|s|^{2\beta}},
$$
 for some $\beta>0$, $C_\beta\geq 0$, from Definition~\ref{defi_Omega}, we get
\begin{equation}\label{Eq_Composition_positive_6}
|S_L^{-1}(\xi,T)g_s(\xi)f(s)|\leq\frac{4C_\varphi C_\beta}{|e^{J\varphi}-e^{J\alpha\varphi'}||1+e^{2J\varphi'}|}\frac{1}{|\xi|}\frac{(|s||\xi|^2+|\xi|^\alpha)}{|s|(|s|+|\xi|^\alpha)(1+|\xi|^2)}\frac{|s|^\beta}{1+|s|^{2\beta}}.
\end{equation}
From this estimate, we can then conclude easily the finiteness of the double integral
\begin{equation*}
\int_{\partial D_\varphi\cap\mathbb{C}_J}\int_{\partial D_{\varphi'}\cap\mathbb{C}_J}|S_L^{-1}(\xi,T)g_s(\xi)f(s)|ds_Jd\xi_J<\infty.
\end{equation*}
Hence we are in the position to apply Fubini's theorem, change the order of integration in \eqref{Eq_Composition_positive_2}, and get
\begin{align*}
f(p_\alpha(T))&=\frac{1}{4\pi^2}\int_{\partial D_{\varphi'}\cap\mathbb{C}_J}S_L^{-1}(\xi,T)d\xi_J\bigg(\int_{\partial D_\varphi\cap\mathbb{C}_J}g_s(\xi)ds_Jf(s)\bigg) \\
&=\frac{1}{4\pi^2}\int_{\partial D_{\varphi'}\cap\mathbb{C}_J}S_L^{-1}(\xi,T)d\xi_J\bigg(\int_{\partial D_\varphi\cap\mathbb{C}_J}\frac{1}{s-p_\alpha(\xi)}ds_Jf(s)\bigg),
\end{align*}
where $s \in \mathbb{C}_J$ and where in the second line we again used that the $\frac{1}{s}$-part of $g_s$ vanishes, as shown in \eqref{Eq_Composition_positive_3}. Since we have chosen $\alpha\varphi'<\varphi$, there is $p_\alpha(\xi)\in D_\varphi\cap\mathbb{C}_J$ for every $\xi\in\partial D_{\varphi'}\cap\mathbb{C}_J$. Consequently, the Cauchy integral formula then gives
\begin{equation}\label{Eq_Composition_positive_5}
f(p_\alpha(T))=\frac{1}{2\pi}\int_{\partial D_{\varphi'}\cap\mathbb{C}_J}S_L^{-1}(\xi,T)d\xi_Jf(p_\alpha(\xi))=(f\circ p_\alpha)(T).
\end{equation}
In the \textit{second step}, we assume that
$$
\widetilde{f}(s):=f(s)-\frac{1}{1+s^2}f_0\in\mathcal{SH}_L^0(D_\theta).
$$
 Then, by \eqref{Eq_Extended}, the extended $\omega$-functional calculus of $f(p_\alpha(T))$ is given by
\begin{equation}\label{Eq_Composition_positive_8}
f(p_\alpha(T))=(1+p_\alpha(T)^2)^{-1}f_0+\widetilde{f}(p_\alpha(T)).
\end{equation}
From the first step, we already know that $\widetilde{f}(p_\alpha(T))=(\widetilde{f}\circ p_\alpha)(T)$. Moreover, by the product rule of the $H^\infty$-functional calculus in Theorem~\ref{thm_Product_rule_positive_powers} there also is
\begin{equation*}
(1+p_\alpha(T)^2)\Big(\frac{1}{1+p_\alpha^2}\Big)(T)=(1+p_\alpha^2)(T)\Big(\frac{1}{1+p_\alpha^2}\Big)(T)=\Big((1+p_\alpha^2)\frac{1}{1+p_\alpha^2}\Big)(T)=1.
\end{equation*}
Since we know that $1+p_\alpha(T)^2$ is bijective, we conclude from this equation, that
\begin{equation}\label{Eq_Composition_positive_9}
(1+p_\alpha(T)^2)^{-1}=\Big(\frac{1}{1+p_\alpha^2}\Big)(T).
\end{equation}
Plugging now \eqref{Eq_Composition_positive_9} into \eqref{Eq_Composition_positive_8}, gives
\begin{equation*}
f(p_\alpha(T))=\Big(\frac{1}{1+p_\alpha^2}\Big)(T)f_0+(\widetilde{f}\circ p_\alpha)(T)=\Big(\frac{1}{1+p_\alpha^2}f_0+\widetilde{f}\circ p_\alpha\Big)(T)=(f\circ p_\alpha)(T).
\end{equation*}
In the \textit{third step} we now combine \eqref{Eq_Composition_positive_5} and \eqref{Eq_Composition_positive_6} to get the composition rule of the $H^\infty$-functional calculus. Since $f$ satisfies the estimate \eqref{Eq_Composition_rule_not_injective}, the function $e(s)=\frac{1}{(1+s^2)^n}$, with $n>\frac{\delta}{2}$, is a possible regularizer. From the second step we already know that
$e(p_\alpha(T))=(e\circ p_\alpha)(T)$ as well as $$(ef)(p_\alpha)(T)=((ef)\circ p_\alpha)(T).$$
 Hence there is
\begin{align}
f(p_\alpha(T))&=e(p_\alpha(T))^{-1}(ef)(p_\alpha(T)) \notag \\
&=(e\circ p_\alpha)(T)^{-1}((ef)\circ p_\alpha)(T) \notag \\
&=(e\circ p_\alpha)(T)^{-1}((e\circ p_\alpha)(f\circ p_\alpha))(T) \notag \\
&=(f\circ p_\alpha)(T), \label{Eq_Composition_positive_7}
\end{align}
where in the last line we used that $e\circ p_\alpha$ is a regularizer of $f\circ p_\alpha$.
\end{proof}

\begin{thm}\label{thm_Composition_positive_powers}
Let $T\in\mathcal{K}(V)$ be bisectorial of angle $\omega\in(0,\frac{\pi}{2})$. Then for every $0<\alpha<\frac{\pi}{2\omega}$ and $\beta>0$, we get \medskip

\begin{minipage}{0.49\textwidth}
\begin{enumerate}
\item[i)] $p_\beta(p_\alpha(T))=p_{\beta\alpha}(T)$, \medskip

\item[ii)] $q_\beta(p_\alpha(T))=q_{\beta\alpha}(T)$,
\end{enumerate}
\end{minipage}
\begin{minipage}{0.49\textwidth}
\begin{enumerate}
\item[iii)] $p_\beta(q_\alpha(T))=q_{\beta\alpha}(T)$, \medskip

\item[iv)] $q_\beta(q_\alpha(T))=q_{\beta\alpha}(T)$.
\end{enumerate}
\end{minipage}
\end{thm}

\begin{proof}
Since both functions $p_\beta$ and $q_\beta$ satisfy the bound \eqref{Eq_Composition_rule_not_injective}, all four identities follow immediately from Proposition~\ref{prop_Composition_rule_not_injective}, together with the identities \medskip

\begin{minipage}{0.49\textwidth}
\begin{enumerate}
\item[i)] $p_\beta(p_\alpha(s))=p_{\beta\alpha}(s)$, \medskip

\item[ii)] $q_\beta(p_\alpha(s))=q_{\beta\alpha}(s)$,
\end{enumerate}
\end{minipage}
\begin{minipage}{0.49\textwidth}
\begin{enumerate}
\item[iii)] $p_\beta(q_\alpha(s))=q_{\beta\alpha}(s)$, \medskip

\item[iv)] $q_\beta(q_\alpha(s))=q_{\beta\alpha}(s)$,\end{enumerate}
\end{minipage}

\medskip of the functions in \eqref{Eq_palpha_qalpha}.
\end{proof}

\subsection{Negative fractional powers}\label{sec_Negative_fractional_powers}

If we want to consider fractional powers with negative exponent, we are facing the problem that the functions $p_\alpha$ and $q_\alpha$ in \eqref{Eq_palpha_qalpha} do not vanish at $s=0$ anymore, since they have a polynomial singularity. For the $H^\infty$-functional calculus in Definition~\ref{defi_Hinfty} this means that we need to regularize at the origin, which by Remark~\ref{rem_Polynomially_growing_functions} means that we need $T$ to be injective.

\begin{defi}\label{defi_palphaT_negative}
Let $T\in\mathcal{K}(V)$ be injective and bisectorial. Then for every $\alpha\in\mathbb{R}$, we define the operators
\begin{equation*}
p_\alpha(T)\text{ and }q_\alpha(T)\text{ via the }H^\infty\text{-functional calculus \eqref{Eq_Hinfty}}.
\end{equation*}
A regularizer is for every integer $n>|\alpha|$ given by
\begin{equation}\label{Eq_Regularizer_negative}
e(s)=\frac{s^n}{(1+s^2)^n}.
\end{equation}
\end{defi}

In our current setting, where we assume that the operator $T$ is injective, we find a similar result as in Proposition~\ref{prop_Product_rule_positive} about when the product rule of the $H^\infty$-functional calculus holds with equality. Namely, in the injective setting we can also allow functions $g,f$ which are vanishing at $\infty$ and are polynomially growing at $0$, and for which $\frac{1}{g}$ vanishes at $0$ and polynomially grows at $\infty$.

\begin{prop}\label{prop_Product_rule_negative}
Let $T\in\mathcal{K}(V)$ be injective and bisectorial of angle $\omega\in(0,\frac{\pi}{2})$, and $f\in\mathcal{SH}_L(D_\theta)$, $g\in\mathcal{N}(D_\theta)$, for some $\theta\in(\omega,\frac{\pi}{2})$. If there exist $0<\delta\leq\gamma$, as well as constants $0<C_1\leq C_2$, such that for some $f_0\in\mathbb{R}_n$, $g_0\in\mathbb{R}$ there hold the estimates
\begin{subequations}
\begin{align}
C_1\min\Big\{\frac{1}{|s|^\delta},\frac{1}{|s|^\gamma}\Big\}\leq|g(s)-g_0|\leq C_2\max\Big\{\frac{1}{|s|^\delta},\frac{1}{|s|^\gamma}\Big\},\qquad s\in D_\theta, \label{Eq_Product_rule_negative_g} \\
|f(s)-f_0|\leq C_2\max\Big\{\frac{1}{|s|^\delta},\frac{1}{|s|^\gamma}\Big\},\qquad s\in D_\theta. \label{Eq_Product_rule_negative_f}
\end{align}
\end{subequations}
Then we have the product rule
\begin{equation*}
(gf)(T)=g(T)f(T).
\end{equation*}
\end{prop}

\begin{proof}
For the same reason as in \eqref{Eq_Product_rule_positive_2}, it is sufficient to verify the domain inclusion
\begin{equation}\label{Eq_Product_rule_negative_2}
\dom((gf)(T))\subseteq\dom(f(T)).
\end{equation}
To do so, let $v\in\dom((gf)(T))$. Due to the upper bound of $g$ in \eqref{Eq_Product_rule_negative_g} and the same one for $f$ in \eqref{Eq_Product_rule_negative_f}, it is possible to consider the regularizer function
\begin{equation*}
e(s)=\frac{s^{2n}}{(1+s^2)^n},
\end{equation*}
with $n>\frac{\delta}{2}$, for both functions $g$ and $f$. Then $e^2$ is a regularizer of the product $gf$ and the condition $v\in\dom((gf)(T))$ translates into
\begin{equation*}
(e^2gf)(T)v\in\ran(e^2(T))=\ran\big(T^{4n}(1+T^2)^{-2n}\big)=\ran(T^{4n}).
\end{equation*}
Next, due to the lower bound \eqref{Eq_Product_rule_negative_g} of $g$, and by the choice $n>\frac{\delta}{2}$, there is $\frac{e}{s^{2n}g}\in\mathcal{N}^0(D_\theta)$, a function for which we can apply the $\omega$-functional calculus. From the operator inclusion in \cite[Corollary~3.18~ii)]{MS24}, it then follows that also
\begin{equation*}
\Big(\frac{e}{s^{2n}g}\Big)(T)(e^2gf)(T)v\in\ran(T^{4n}).
\end{equation*}
However, with the product rule \cite[Theorem~3.11]{MS24} of the $\omega$-functional calculus, we can rewrite this expression as
\begin{equation*}
\Big(\frac{e}{s^{2n}g}\Big)(T)(e^2gf)(T)v=\Big(\frac{e^3f}{s^{2n}}\Big)(T)v=\Big(\frac{e^2}{s^{2n}}\Big)(T)(ef)(T)v.
\end{equation*}
Altogether, this means that $$(\frac{e^2}{s^{2n}})(T)(ef)(T)v\in\ran(T^{4n}),$$ and since we can explicitly write $$(\frac{e^2}{s^{2n}})(T)=T^{2n}(1+T^2)^{-2n},$$ this implies
\begin{equation*}
(ef)(T)v\in\ran(T^{2n})=\ran(e(T)).
\end{equation*}
This proves $v\in\dom(f(T))$, and we have verified the domain inclusion \eqref{Eq_Product_rule_negative_2}.
\end{proof}

The next theorem proves the power rule of fractional powers for exponents which are both positive or both negative. The same product rule for mixed powers cannot hold because the exponent of the sum is smaller than the individual ones, and hence the domain will increase. The power rule for mixed powers will be investigated in Theorem~\ref{thm_Product_rule_mixed_powers} under the additional assumption that $T$ has dense domain and range.

\begin{thm}\label{thm_Product_rule_negative_powers}
Let $T\in\mathcal{K}(V)$ be injective and bisectorial. Then for every $\alpha,\beta\in\mathbb{R}\setminus\{0\}$ both positive or both positive,
we have the power rules \medskip

\begin{enumerate}
\item[i)] $p_{\alpha+\beta}(T)=p_\alpha(T)q_\beta(T)=q_\alpha(T)p_\beta(T)$; \medskip

\item[ii)] $q_{\alpha+\beta}(T)=p_\alpha(T)p_\beta(T)=q_\alpha(T)q_\beta(T)$.
\end{enumerate}
\end{thm}

\begin{proof}
If $\alpha,\beta$ are both positive, the statement is already proven in Theorem~\ref{thm_Product_rule_positive_powers}. If $\alpha,\beta$ are both negative, the functions $p_\alpha$, $q_\alpha$, $p_\beta$ and $q_\beta$ satisfy the bounds \eqref{Eq_Product_rule_negative_g} and \eqref{Eq_Product_rule_negative_f}. Since for every $\Sc(s)\neq 0$, there also hold the function identities \eqref{Eq_Product_rule_positive_powers_1}, the stated product rules are a consequence of the product rule in Proposition~\ref{prop_Product_rule_negative}.
\end{proof}

The following corollary is a direct consequence of the product rule of fractional powers in Theorem~\ref{thm_Product_rule_negative_powers}.

\begin{cor}\label{cor_Operator_domain_negative_powers}
Let $T\in\mathcal{K}(V)$ be injective and bisectorial, and $\alpha,\beta\in\mathbb{R}\setminus\{0\}$ both positive or both negative. If $|\beta|>|\alpha|$, then \medskip

\begin{enumerate}
\item[i)] $\dom(p_\beta(T))\subseteq\dom(p_\alpha(T))\cap\dom(q_\alpha(T))$; \medskip

\item[ii)] $\dom(q_\beta(T))\subseteq\dom(p_\alpha(T))\cap\dom(q_\alpha(T))$.
\end{enumerate}
\end{cor}

\begin{rem}
Note that in the domain inclusions of Corollary~\ref{cor_Operator_domain_negative_powers} we really need $|\beta|>|\alpha|$. In the case $\alpha=\beta\in\mathbb{R}\setminus\{0\}$, we in general have neither inclusion
\begin{equation*}
\dom(p_\alpha(T))\substack{\not\subset \\ \not\supset}\dom(q_\alpha(T)).
\end{equation*}
\end{rem}

The next theorem is the counterpart to Theorem~\ref{thm_palpha_bisectorial_positive} for negative powers. Also the special case $\alpha=0$ is covered in this setting. Compared to Theorem~\ref{thm_palpha_bisectorial_positive} we need to assume that $T$ is injective so that $p_\alpha(T)$ is defined in the first place.

\begin{thm}
Let $T\in\mathcal{K}(V)$ be injective and bisectorial of angle $\omega\in(0,\frac{\pi}{2})$.

\begin{enumerate}
\item[i)] If $\alpha\in(-\frac{\pi}{2\omega},\frac{\pi}{2\omega})\setminus\{0\}$, then $p_\alpha(T)$ is bisectorial of angle $|\alpha|\omega$. \medskip

\item[ii)] If $\alpha\in(-\frac{\pi}{\omega},\frac{\pi}{\omega})\setminus\{0\}$, then $q_\alpha(T)$ is sectorial of angle $|\alpha|\omega$.
\end{enumerate}
\end{thm}

\begin{proof}
For positive powers $\alpha>0$, the statement is already proven in Theorem~\ref{thm_palpha_bisectorial_positive}. Hence it is left to consider $\alpha<0$. \medskip

i)\;\;The proof of the $S$-spectrum being included in the double sector
\begin{equation*}
\sigma_S(p_\alpha(T))\subseteq\overline{D_{|\alpha|\omega}},
\end{equation*}
is the same as in the proof of Theorem~\ref{thm_palpha_bisectorial_positive}. The main difference is that, for negative powers $\alpha<0$, the function $h_s\overline{s}-p_\alpha h_s$, which represents $S_L^{-1}(s,p_\alpha(T))$ in \eqref{Eq_palpha_bisectorial_5}, has the limits $0$ as $\xi\rightarrow 0$ and $\frac{1}{s}$ as $\xi\rightarrow\infty$. Hence we need to extract the limit $\frac{1}{s}$ at $\xi=\infty$ from this function, in order to apply the $\omega$-functional calculus. This means, we decompose the $S$-resolvent operator into
\begin{equation}\label{Eq_palpha_bisectorial_negative_2}
S_L^{-1}(s,p_\alpha(T))=T^2(|s|^{\frac{2}{\alpha}}+T^2)^{-1}\frac{1}{s}+\Big(\underbrace{h_s(\xi)\overline{s}-p_\alpha(\xi)h_s(\xi)-\frac{\xi^2}{|s|^{\frac{2}{\alpha}}+\xi^2}\frac{1}{s}}_{=:g_s(\xi)}\Big)(T).
\end{equation}
The remaining task is now to estimate the operator $g_s(T)$, defined via the $\omega$-functional calculus of Definition~\ref{defi_Omega}. First, for every $\xi\in\overline{D_{\varphi'}}$, we can estimate
\begin{align*}
|g_s(\xi)|&=\bigg|\frac{1}{s-p_\alpha(\xi)}-\frac{\xi^2}{s(|s|^{\frac{2}{\alpha}}+\xi^2)}\bigg|\leq\frac{|s|^{1+\frac{2}{\alpha}}+|\xi|^{2+\alpha}}{|s||s-p_\alpha(\xi)|||s|^{\frac{2}{\alpha}}+\xi^2|} \\
&\leq\frac{4}{|e^{J\varphi}-e^{J|\alpha|\varphi'}||1+e^{2J\varphi'}|}\frac{|s|^{1+\frac{2}{\alpha}}+|\xi|^{2+\alpha}}{|s|(|s|+|\xi|^\alpha)(|s|^{\frac{2}{\alpha}}+|\xi|^2)},
\end{align*}
where in the first equality we used that $p_\alpha(\xi)\in\mathbb{C}_J$ is in the same complex plane as $s$ and $\xi \in \mathbb{C}_J$, and hence the term $h_s(\xi)\overline{s}-p_{-\alpha}(\xi)h_s(\xi)$ reduces to $(s-p_\alpha(\xi))^{-1}$. Moreover, in the second line, we used the same estimates as in \eqref{Eq_palpha_bisectorial_7} and \eqref{Eq_palpha_bisectorial_8}. Now that we have estimated the function $g_s$, we can combine it with the bound \eqref{Eq_SL_estimate} of the $S$-resolvent of $T$, to get an upper bound of the $\omega$-functional calculus
\begin{align}
\Vert g_s(T)\Vert&=\frac{1}{2\pi}\bigg\Vert\int_{\partial D_{\varphi'}\cap\mathbb{C}_J}S_L^{-1}(\xi,T)d\xi_Jg_s(\xi)\bigg\Vert \notag \\
&\leq\frac{8C_{\varphi'}}{\pi|e^{J\varphi}-e^{J|\alpha|\varphi'}||1+e^{2J\varphi'}|}\int_0^\infty\frac{|s|^{1+\frac{2}{\alpha}}+r^{2+\alpha}}{r|s|(|s|+r^\alpha)(|s|^{\frac{2}{\alpha}}+r^2)}dr \notag \\
&=\frac{8C_{\varphi'}}{\pi|e^{J\varphi}-e^{J|\alpha|\varphi'}||1+e^{2J\varphi'}||s|}\int_0^\infty\frac{1+r^{2+\alpha}}{r(1+r^\alpha)(1+r^2)}dr \notag \\
&=\frac{8(1+\frac{2}{|\alpha|})C_{\varphi'}}{\pi|e^{J\varphi}-e^{J|\alpha|\varphi'}||1+e^{2J\varphi'}||s|}, \label{Eq_palpha_bisectorial_negative_1}
\end{align}
where in the third line we substituted $r\rightarrow|s|^{\frac{1}{\alpha}}r$, and in the last we used $1+r^\alpha\geq\max\{1,r^\alpha\}$ as well as $1+r^2\geq\max\{1,r^2\}$ in order to compute the integral explicitly. Finally, since we can use the pseudo resolvent operator \eqref{Eq_Qs}, to write
\begin{equation*}
T^2(|s|^{\frac{2}{\alpha}}+T^2)^{-1}=1-|s|^{\frac{2}{\alpha}}Q_{|s|^{\frac{1}{\alpha}}J}[T]^{-1},
\end{equation*}
we can combine the estimate \cite[Lemma 3.2 ii)]{CMS25} of the operator $Q_{|s|^{\frac{1}{\alpha}}J}[T]^{-1}$ with the estimate \eqref{Eq_palpha_bisectorial_negative_1}, to estimate the full $S$-resolvent operator \eqref{Eq_palpha_bisectorial_negative_2} by
\begin{align*}
\Vert S_L^{-1}(s,p_\alpha(T))\Vert&\leq\frac{1}{|s|}\Big\Vert 1-|s|^{\frac{2}{\alpha}}Q_{|s|^{\frac{1}{\alpha}}J}[T]^{-1}\Big\Vert+\Vert g_s(T)\Vert \\
&\leq\bigg(1+2C_\varphi^2+\frac{8(1+\frac{2}{|\alpha|})C_{\varphi'}}{\pi|e^{J\varphi}-e^{J|\alpha|\varphi'}||1+e^{2J\varphi'}|}\bigg)\frac{1}{|s|}.
\end{align*}
ii)\;\;The proof of $q_\alpha(T)$ being sectorial follows the same arguments, with the minor changes mentioned in the proof of Theorem~\ref{thm_palpha_bisectorial_positive}~ii).
\end{proof}

\begin{prop}
Let $T\in\mathcal{K}(V)$ be injective and bisectorial. Then for every $\alpha\in\mathbb{R}$, $p_\alpha(T)$ and $q_\alpha(T)$ are injective, with
\begin{equation*}
p_\alpha(T)^{-1}=p_{-\alpha}(T)\qquad\text{and}\qquad q_\alpha(T)^{-1}=q_{-\alpha}(T).
\end{equation*}
\end{prop}

\begin{proof}
We will only consider $p_\alpha(T)$, the proof for $q_\alpha(T)$ is analog. For the inclusion $\text{\grqq}\subseteq\text{\grqq}$ we know that by the definition of the $H^\infty$-functional calculus, with the regularizer $e$ from \eqref{Eq_Regularizer_negative}, there is
\begin{equation}\label{Eq_palpha_injective_1}
(ep_{-\alpha})(T)p_\alpha(T)=(ep_{-\alpha})(T)e(T)^{-1}(ep_\alpha)(T)\subseteq e(T)^{-1}(ep_{-\alpha}ep_\alpha)(T)=e(T).
\end{equation}
Since $e(T)$ is injective, this equation shows on the one hand that $p_\alpha(T)$ is injective. This equation also shows that $p_\alpha(T)v\in\dom(p_{-\alpha}(T))$, for every $v\in\dom(p_\alpha(T))$, and $p_{-\alpha}(T)p_\alpha(T)v=v$. In other words this proves the inclusion $p_\alpha(T)^{-1}\subseteq p_{-\alpha}(T)$. \medskip

For the inverse inclusion $\text{\grqq}\supseteq\text{\grqq}$, we obtain
\begin{equation*}
(ep_\alpha)(T)p_{-\alpha}(T)\subseteq e(T),
\end{equation*}
in the same way as in \eqref{Eq_palpha_injective_1}. This proves that $p_{-\alpha}(T)v\in\dom(p_\alpha(T))$ for every $v\in\dom(p_{-\alpha}(T))$, and $p_\alpha(T)p_{-\alpha}(T)v=v$. Consequently, $v\in\dom(p_{\alpha}(T)^{-1})$.
\end{proof}

Next, under the additional assumption that $T$ is injective, we extend the composition rule of Proposition~\ref{prop_Composition_rule_not_injective} also to negative values of $\alpha$ and $\beta$.

\begin{prop}\label{prop_Composition_rule_injective}
Let $T\in\mathcal{K}(V)$ be injective and bisectorial of angle $\omega\in(0,\frac{\pi}{2})$, and $\alpha\in(-\frac{\pi}{2\omega},\frac{\pi}{2\omega})\setminus\{0\}$. Consider a function $f\in\mathcal{SH}_L(D_\theta)$, with $\theta\in(\alpha\omega,\frac{\pi}{2}))$, which satisfies the estimate
\begin{equation}\label{Eq_Composition_rule_injective}
|f(s)|\leq C\Big(\frac{1}{|s|^\gamma}+|s|^\gamma\Big),\qquad s\in D_\theta,
\end{equation}
for some $\gamma>0$ and $C\geq 0$. Then $f\circ p_\alpha\in\mathcal{SH}_L(D_{\theta'})$, for every $\theta'<\min\{\frac{\theta}{|\alpha|},\frac{\pi}{2}\}$ and there holds \medskip

\begin{enumerate}
\item[i)] $f(p_\alpha(T))=(f\circ p_\alpha)(T)$, \medskip

\item[ii)] $f(q_\alpha(T))=(f\circ q_\alpha)(T)$.
\end{enumerate}
\end{prop}

\begin{proof}
The case $\alpha>0$ is already considered in Proposition~\ref{prop_Composition_rule_not_injective}. Hence we will only consider the case $\alpha<0$. We will also only prove i), then ii) follows similarly. We will do the proof in two steps, first for the $\omega$-functional calculus \eqref{Eq_Omega}, and then for the $H^\infty$-functional calculus \eqref{Eq_Hinfty}. \medskip

In the \textit{first step} let us assume that $f\in\mathcal{SH}_L^0(D_\theta)$. Then the proof of
\begin{equation*}
f(p_\alpha(T))=(f\circ p_\alpha)(T)
\end{equation*}
works in the same way as in the proof of Proposition~\ref{prop_Composition_rule_not_injective}, with the only difference that we use the representation
\begin{equation*}
S_L^{-1}(s,p_\alpha(T))=T^2(1+T^2)^{-1}\frac{1}{s}+\Big(\underbrace{h_s(\xi)\overline{s}-p_\alpha(\xi)h_s(\xi)-\frac{\xi^2}{1+\xi^2}\frac{1}{s}}_{=:g_s(\xi)}\Big)(T),
\end{equation*}
of the $S$-resolvent operator, instead of \eqref{Eq_Composition_positive_4}, which then leads to the estimate
\begin{equation*}
|S_L^{-1}(\xi,T)g_s(\xi)f(s)|\leq\frac{4C_\varphi C_\beta}{|e^{J\varphi}-e^{J|\alpha|\varphi'}||1+e^{2J\varphi'}|}\frac{1}{|\xi|}\frac{|s|+|\xi|^{2+\alpha}}{|s|(|s|+|\xi|^\alpha)(1+|\xi|^2)}\frac{|s|^\beta}{1+|s|^{2\beta}},
\end{equation*}
instead of \eqref{Eq_Composition_positive_6}. \medskip

In the \textit{second step}, we will now prove the composition rule for the $H^\infty$-functional calculus. To do so, let $e$ be as in \eqref{Eq_Regularizer_negative} with $n>\gamma$, a regularizer of $f$. Then we know from the first step that $e(p_\alpha(T))=(e\circ p_\alpha)(T)$ and $(ef)(p_\alpha(T))=((ef)\circ p_\alpha)(T)$, and as in \eqref{Eq_Composition_positive_7}, we then prove
\begin{equation*}
f(p_\alpha(T))=(f\circ p_\alpha)(T). \qedhere
\end{equation*}
\end{proof}

\begin{thm}
Let $T\in\mathcal{K}(V)$ be injective and bisectorial of angle $\omega\in(0,\frac{\pi}{2})$. Then for every $\alpha\in(-\frac{\pi}{2\omega},\frac{\pi}{2\omega})\setminus\{0\}$, and every $\beta\in\mathbb{R}$, we get \medskip

\begin{minipage}{0.49\textwidth}
\begin{enumerate}
\item[i)] $p_\beta(p_\alpha(T))=p_{\beta\alpha}(T)$, \medskip

\item[ii)] $q_\beta(p_\alpha(T))=q_{\beta\alpha}(T)$,
\end{enumerate}
\end{minipage}
\begin{minipage}{0.49\textwidth}
\begin{enumerate}
\item[iii)] $p_\beta(q_\alpha(T))=q_{\beta\alpha}(T)$, \medskip

\item[iv)] $q_\beta(q_\alpha(T))=q_{\beta\alpha}(T)$.
\end{enumerate}
\end{minipage}
\end{thm}

\begin{proof}
Since both functions $p_\beta$ and $q_\beta$ satisfy the bound \eqref{Eq_Composition_rule_injective}, all four identities follow immediately from Proposition~\ref{prop_Composition_rule_injective} as well as the function identities i)--iv) in the proof of Theorem~\ref{thm_Composition_positive_powers}.
\end{proof}

While the product rules for fractional powers in Theorem~\ref{thm_Product_rule_negative_powers} only considers two positive or two negative powers, a respective product rule for mixed powers is missing. Since by Corollary~\ref{cor_Operator_domain_negative_powers} the operator domains of $p_\alpha(T)$ and $q_\alpha(T)$ become smaller when $|\alpha|$ increases, a product rule of the form in Theorem~\ref{thm_Product_rule_negative_powers} cannot hold true for arbitrary $\alpha,\beta\in\mathbb{R}$. However, we will show in Theorem~\ref{thm_Product_rule_mixed_powers} that there is a closure needed on the right hand side to make it an equality. But first, we need the following preparatory lemma.

\begin{lem}\label{lem_Core}
Let $T\in\mathcal{K}(V)$ be injective and bisectorial of angle $\omega\in(0,\frac{\pi}{2})$, with $\overline{\dom}(T)=\overline{\ran}(T)=V$. Consider also a function $f\in\mathcal{N}(D_\theta)$, $\theta\in(\omega,\frac{\pi}{2})$, which is bounded by
\begin{equation}\label{Eq_Core_bound}
|f(s)|\leq C\Big(\frac{1}{|s|^k}+|s|^k\Big),\qquad s\in D_\theta,
\end{equation}
for some $k\in\mathbb{N}$, $C\geq 0$. Then, for every integer $m\geq k+1$, we have
\begin{equation*}
\dom(T^m)\cap\ran(T^m)\subseteq\dom(f(T)),
\end{equation*}
and for every element $v\in\dom(f(T))$, there exists a sequence $(v_n)_n\in\dom(T^m)\cap\ran(T^m)$, such that
\begin{equation*}
v=\lim\limits_{n\rightarrow\infty}v_n\qquad\text{and}\qquad f(T)v=\lim\limits_{n\rightarrow\infty}f(T)v_n.
\end{equation*}
\end{lem}

\begin{proof}
Let us choose the function $e(s)=\frac{s^m}{(1+s^2)^m}$, which is, since $m\geq k+1$, a regularizer of the function $f$. Then we know that $e(T)=T^m(1+T^2)^{-m}$ by \cite[Theorem~3.16]{MS24}, and since we explicitly know the range of this operator, we get
\begin{equation*}
\dom(T^m)\cap\ran(T^m)=\ran(e(T))\subseteq\dom(f(T)),
\end{equation*}
where the last inclusion follows from $e(T)(ef)(T)=(ef)(T)e(T)$, see \cite[Corollary~3.18~iii)]{MS24}. Next, consider for every $n\in\mathbb{N}$ the bounded operator
\begin{equation*}
r_n[T]:=n^2T^2(T^2+n^2)^{-1}\Big(T^2+\frac{1}{n^2}\Big)^{-1}.
\end{equation*}
Choosing now any integer $\widetilde{m}\geq\frac{m}{2}$, we can prove as in \cite[Eq.(3.17)]{CMS25}, that there converges
\begin{equation}\label{Eq_Core_2}
\lim\limits_{n\rightarrow\infty}r_n[T]^{\widetilde{m}}v=v,\qquad v\in V,
\end{equation}
where $\overline{\dom}(T)\cap\overline{\ran}(T)=V$ by the assumption of our lemma. Since we can write the $\widetilde{m}$-th power of $r_n[T]$ as
\begin{equation*}
r_n[T]^{\widetilde{m}}=n^{2\widetilde{m}}T^{2\widetilde{m}}(T^2+n^2)^{-\widetilde{m}}\Big(T^2+\frac{1}{n^2}\Big)^{-\widetilde{m}},
\end{equation*}
it is obvious that
\begin{equation*}
\ran(r_n[T]^{\widetilde{m}})=\dom(T^{2\widetilde{m}})\cap\ran(T^{2\widetilde{m}})\subseteq\dom(T^m)\cap\ran(T^m).
\end{equation*}
Let now $v\in\dom(f(T))$. If we choose
\begin{equation*}
v_n:=r_n[T]^{\widetilde{m}}v\in\dom(T^m)\cap\ran(T^m),
\end{equation*}
then by \eqref{Eq_Core_2} we obtain the convergence
\begin{equation*}
v=\lim\limits_{n\rightarrow\infty}v_n.
\end{equation*}
Moreover, since there commutes $Tr_n[T]\supseteq r_n[T]T$, we can use the convergence \eqref{Eq_Core_2} with the vector $Tv$ instead of $v$, to get also the convergence
\begin{equation*}
Tv=\lim\limits_{n\rightarrow\infty}r_n[T]^{\widetilde{m}}Tv=\lim\limits_{n\rightarrow\infty}Tr_n[T]^{\widetilde{m}}v=\lim\limits_{n\rightarrow\infty}Tv_n. \qedhere
\end{equation*}
\end{proof}

The next theorem is another improvement of the classical product rule \cite[Theorem~5.7]{MS24} of the $H^\infty$-functional calculus, in the sense that the operator inclusion holds with equality when one takes the closure of the right hand side. Compared to Proposition~\ref{prop_Product_rule_positive} and Proposition~\ref{prop_Product_rule_negative}, we do not need $f$ and $g$ to be vanishing either at $0$ or at $\infty$. The improvement comes from the additional assumption that $T$ has dense domain and range.

\begin{prop}\label{prop_Product_rule_mixed}
Let $T\in\mathcal{K}(V)$ be injective and bisectorial of angle $\omega\in(0,\frac{\pi}{2})$, with $\overline{\dom}(T)=\overline{\ran}(T)=V$. Let $f,g\in\mathcal{N}(D_\theta)$, for some $\theta\in(\omega,\frac{\pi}{2})$, both satisfy \eqref{Eq_Core_bound}. Then $g(T)f(T)$ is closable, and there holds the product rule
\begin{equation}\label{Eq_Product_rule_mixed}
(gf)(T)=\overline{g(T)f(T)}.
\end{equation}
\end{prop}

\begin{proof}
According to \cite[Theorem~5.7]{MS24}, there holds the inclusion
\begin{equation*}
(gf)(T)\supseteq g(T)f(T).
\end{equation*}
Since $(gf)(T)$ is also closed by \cite[Lemma~5.5~i)]{MS24}, we have proven $\text{\grqq}\supseteq\text{\grqq}$ in \eqref{Eq_Product_rule_mixed}. \medskip

For the inclusion $\text{\grqq}\subseteq\text{\grqq}$, we will first prove that with the function
$$
e(s)=\frac{s^{2k+1}}{(1+s^2)^{2k+1}},
$$
 which is a regularizer of $f$, $g$ and of the product $gf$, there is
\begin{equation}\label{Eq_Product_rule_mixed_1}
\overline{(fe)(T)e(T)^{-1}}=f(T).
\end{equation}
By the commutation property \cite[Corollary~3.18~iii)]{MS24} of $e(T)$ and $(fe)(T)$, we obtain the operator inclusion
\begin{equation}\label{Eq_Product_rule_mixed_2}
(fe)(T)e(T)^{-1}=e(T)^{-1}(fe)(T)e(T)e(T)^{-1}\subseteq e(T)^{-1}(fe)(T)=f(T).
\end{equation}
Since $f(T)$ is also closed by \cite[Lemma~5.5]{MS24}, we conclude the first inclusion in \eqref{Eq_Product_rule_mixed_1}. \medskip

For the inverse inclusion, let $v\in\dom(f(T))$. Then by Lemma~\ref{lem_Core}, there exists a sequence $(v_n)_n\in\ran(e(T))=\dom(T^{2k+1})\cap\ran(T^{2k+1})$, such that
\begin{equation*}
\lim\limits_{n\rightarrow\infty}v_n=v\qquad\text{and}\qquad\lim\limits_{n\rightarrow\infty}f(T)v_n=f(T)v.
\end{equation*}
Since $v_n\in\ran(e(T))$, we can apply it to \eqref{Eq_Product_rule_mixed_2}, and write the second limit as
\begin{equation*}
\lim\limits_{n\rightarrow\infty}(fe)(T)e(T)^{-1}v_n=\lim\limits_{n\rightarrow\infty}f(T)v_n=f(T)v.
\end{equation*}
This shows that $v\in\dom\big(\overline{(fe)(T)e(T)^{-1}}\big)$ and $f(T)v=\overline{(fe)(T)e(T)^{-1}}v$. \medskip

Now that we have proven \eqref{Eq_Product_rule_mixed_1}, we can combine this result with the operator $g(T)$. First, since $(gfe)(T)$ is a bounded operator, we get
\begin{equation*}
(gfe)(T)e(T)^{-1}=g(T)(fe)(T)e(T)^{-1}\subseteq g(T)\overline{(fe)(T)e(T)^{-1}}=g(T)f(T),
\end{equation*}
where in the last equality we used \eqref{Eq_Product_rule_mixed_1}. If we now take the closure of this inclusion, and use again \eqref{Eq_Product_rule_mixed_1} with $gf$ instead of $f$, we obtain the remaining operator inclusion
\begin{equation*}
(gf)(T)= \overline{(gfe)(T)e(T)^{-1}}\subseteq\overline{g(T)f(T)}. \qedhere
\end{equation*}
\end{proof}

The next theorem is the remaining third one, finalizing together with Theorem~\ref{thm_Product_rule_positive_powers} and Theorem~\ref{thm_Product_rule_negative_powers} the power rule of the fractional powers. It gives the power rule for all possible exponents $\alpha,\beta\in\mathbb{R}$, but since for $\alpha$ and $\beta$ both being positive or both being negative we have already stronger results in Theorem~\ref{thm_Product_rule_positive_powers} and Theorem~\ref{thm_Product_rule_negative_powers}, the real application for the following result is for mixed positive and negative powers.

\begin{thm}\label{thm_Product_rule_mixed_powers}
Let $T\in\mathcal{K}(V)$ be injective and bisectorial, with $\overline{\dom}(T)=\overline{\ran}(T)=V$. Then for all $\alpha,\beta\in\mathbb{R}$ there hold the power rules \medskip

\begin{enumerate}
\item[i)] $p_{\alpha+\beta}(T)=\overline{p_\alpha(T)q_\beta(T)}=\overline{q_\alpha(T)p_\beta(T)}$ \medskip

\item[ii)] $q_{\alpha+\beta}(T)=\overline{p_\alpha(T)p_\beta(T)}=\overline{q_\alpha(T)q_\beta(T)}$
\end{enumerate}
\end{thm}

\begin{proof}
Since for every $\Sc(s)\neq 0$, there hold the function identities \eqref{Eq_Product_rule_positive_powers_1}, the stated product rules are a consequence of the product rule in Proposition~\ref{prop_Product_rule_mixed}.
\end{proof}

\begin{rem}
If the Banach module $V$ is $\mathbb{R}$-reflexive (or even a Hilbert space), the fact that $T$ is injective and bisectorial, already implies that $\overline{\dom}(T)=\overline{\ran}(T)=V$, see \cite[Theorem~3.3]{CMS25}. In particular, if $V$ is $\mathbb{R}$-reflexive, the statements of Lemma~\ref{lem_Core}, Proposition~\ref{prop_Product_rule_mixed} and Theorem~\ref{thm_Product_rule_mixed_powers} hold without the explicit assumption of dense domain and range.
\end{rem}

\subsection{Connection to fractional powers of $T^2$}

In this subsection, we will find a connection between the two versions $p_\alpha(T)$ and $q_\alpha(T)$ of fractional powers, to the classical fractional powers of the squared operator $T^2$. As classical fractional power we understand in this context the function $f(s)=s^\alpha$, defined for every $s\in\mathbb{R}^{n+1}\setminus(-\infty,0]$, and the corresponding operator $f(T)$ defined via the $H^\infty$-functional calculus of sectorial operators. Consequently, we first show that for any bisectorial operator $T$, the squared operator $T^2$ is sectorial.

\begin{lem}\label{lem_Spectrum_T2}
For every $T\in\mathcal{K}(V)$ with $\rho_S(T)\neq\emptyset$. Then $T^2\in\mathcal{K}(V)$ and its $S$-spectrum is given by
\begin{equation}\label{Eq_Spectrum_T2}
\sigma_S(T^2)=\sigma_S(T)^2.
\end{equation}
Moreover, for every $s\in\mathbb{R}^{n+1}$ with $s^2\in\rho_S(T^2)$, we have
\begin{subequations}
\begin{align}
2S_L^{-1}(s^2,T^2)s&=S_L^{-1}(s,T)-S_L^{-1}(-s,T),\qquad\text{and} \label{Eq_SL_T2} \\
2sS_R^{-1}(s^2,T^2)&=S_R^{-1}(s,T)-S_R^{-1}(-s,T). \label{Eq_SR_T2}
\end{align}
\end{subequations}
\end{lem}

\begin{proof}
First of all, since $\rho_S(T)\neq\emptyset$, it follows from \cite[Proposition~2.7]{ADJOINT} with $s=0$ that $T^2\in\mathcal{K}(V)$. Furthermore, for every $s\in\mathbb{R}^{n+1}$, the operator \eqref{Eq_Qs} satisfies
\begin{align}
Q_s[T]Q_{-s}[T]&=(T^2-2s_0T+|s|^2)(T^2+2s_0T+|s|^2) \notag \\
&=T^4+2(|s|^2-2s_0^2)T^2+|s|^4=Q_{s^2}[T^2]. \label{Eq_Spectrum_T2_1}
\end{align}
Using this equation with $s$ and once with $s$ replaced by $-s$, we have
\begin{equation*}
Q_{s^2}[T^2]\text{ is bijective}\quad\Leftrightarrow\quad Q_s[T],Q_{-s}[T]\text{ are bijective}.
\end{equation*}
By \cite[Remark~2.6]{ADJOINT} this equivalence can now also be written as
\begin{equation*}
s^2\in\rho_S(T^2)\quad\Leftrightarrow\quad s,-s\in\rho_S(T).
\end{equation*}
This proves the accordance of the $S$-spectra in \eqref{Eq_Spectrum_T2}. In order to show the relation between the $S$-resolvents \eqref{Eq_SL_T2}, let $s\in\mathbb{R}^{n+1}$ with $s^2\in\rho_S(T^2)$. Then there is $s,-s\in\rho_S(T)$ and all the pseudo-resolvents in \eqref{Eq_Spectrum_T2_1} are invertible, with
\begin{equation*}
Q_{-s}[T]^{-1}Q_s[T]^{-1}=Q_{s^2}[T^2]^{-1}.
\end{equation*}
Then, by the definition \eqref{Eq_SL_SR} of the left $S$-resolvent, we get
\begin{align*}
S_L^{-1}(s,T)-S_L^{-1}(-s,T)&=\big(Q_s[T]^{-1}+Q_{-s}[T]^{-1}\big)\overline{s}-T\big(Q_s[T]^{-1}-Q_{-s}[T]^{-1}\big) \\
&=2Q_s[T]^{-1}(T^2+|s|^2)Q_{-s}[T]^{-1}\overline{s}-TQ_s[T]^{-1}4s_0TQ_{-s}[T]^{-1} \\
&=2(T^2+|s|^2)Q_{s^2}[T^2]^{-1}\overline{s}-4s_0T^2Q_{s^2}[T^2]^{-1} \\
&=2|s|^2Q_{s^2}[T^2]^{-1}\overline{s}-2T^2Q_{s^2}[T^2]^{-1}s \\
&=\big(Q_{s^2}[T^2]^{-1}\overline{s}^2-T^2Q_{s^2}[T^2]^{-1}\big)2s=S_L^{-1}(s^2,T^2)2s.
\end{align*}
The similar relation \eqref{Eq_SR_T2} for the right $S$-resolvent operator follows analogously.
\end{proof}

With Lemma~\ref{lem_Spectrum_T2} we are now in the position to prove that $T^2$ is sectorial, if $T$ is bisectorial.

\begin{prop}
Let $T\in\mathcal{K}(V)$ be bisectorial of angle $\omega\in(0,\frac{\pi}{2})$. Then $T^2\in\mathcal{K}(V)$ is sectorial of angle $2\omega$.
\end{prop}

\begin{proof}
Since $T$ is bisectorial, there is in particular $\rho_S(T)\neq\emptyset$ and hence $T^2\in\mathcal{K}(V)$ by Lemma~\ref{lem_Spectrum_T2}.
Since $\sigma_S(T)\subseteq\overline{D_\omega}$, we deduce from Lemma~\ref{lem_Spectrum_T2}, that
\begin{equation*}
\sigma_S(T^2)=\sigma_S(T)^2\subseteq(\overline{D_\omega})^2=\overline{S_{2\omega}}.
\end{equation*}
Moreover, let $\varphi\in(2\omega,\pi)$ and consider $q\in\mathbb{R}^{n+1}\setminus(S_{2\varphi}\cup\{0\})$. Then $q=s^2$, for some $s\in\mathbb{R}^{n+1}\setminus(D_\varphi\cup\{0\})$, and using the representation \eqref{Eq_SL_T2} of the left $S$-resolvent operator of $T^2$, we get the estimate
\begin{equation*}
\Vert S_L^{-1}(q,T^2)\Vert\leq\frac{\Vert S_L^{-1}(s,T)\Vert+\Vert S_L^{-1}(-s,T)\Vert}{2|s|}\leq\frac{C_\varphi}{|s|^2}=\frac{C_\varphi}{|q|}.
\end{equation*}
This proves that $T^2$ is a sectorial operator of angle $2\omega$.
\end{proof}

\begin{thm}\label{thm_palpha_qalpha_T2}
Let $T\in\mathcal{K}(V)$ be injective and bisectorial of angle $\omega\in(0,\frac{\pi}{2})$ and $\alpha\in\mathbb{R}$. Then we have the relations
\begin{equation*}
q_\alpha(T)=(T^2)^{\frac{\alpha}{2}}\qquad\text{and}\qquad p_\alpha(T)=\overline{\sgn(T)(T^2)^{\frac{\alpha}{2}}}
\end{equation*}
between $p_\alpha(T)$ and $q_\alpha(T)$ from Definition~\ref{defi_palphaT_negative} and the fractional power $(T^2)^{\frac{\alpha}{2}}$ understood as the sectorial $H^\infty$-functional calculus of the operator $T^2$ and the function $f(s)=s^{\frac{\alpha}{2}}$, see \cite[Definition~7.2.5]{FJBOOK},
where the sign-function is understood as the sign of the scalar part.
\end{thm}

\begin{proof}
Let us first consider the operator $q_\alpha(T)$. Consider the regularizer $e(s)=\frac{s^{2n}}{(1+s^2)^{2n}}$, for some integer $n>\frac{|\alpha|}{2}$. Then the $\omega$-functional calculus \eqref{Eq_Omega} of the product $(eq_\alpha)(T)$ can for any $J\in\mathbb{S}$ and $\varphi\in(\omega,\frac{\pi}{2})$ be written as
\begin{align*}
(eq_\alpha)(T)&=\frac{1}{2\pi}\int_{\partial D_\varphi\cap\mathbb{C}_J}S_L^{-1}(s,T)ds_Je(s)q_\alpha(s) \\
&=\frac{1}{2\pi}\int_{\partial S_\varphi\cap\mathbb{C}_J}S_L^{-1}(s,T)ds_Je(s)q_\alpha(s)-\frac{1}{2\pi}\int_{\partial S_\varphi\cap\mathbb{C}_J}S_L^{-1}(-s,T)ds_Je(-s)q_\alpha(-s) \\
&=\frac{1}{2\pi}\int_{\partial S_\varphi\cap\mathbb{C}_J}\big(S_L^{-1}(s,T)-S_L^{-1}(-s,T)\big)ds_Je(s)q_\alpha(s),
\end{align*}
where in the last line we used that $e(-s)=e(s)$ and $q_\alpha(-s)=q_\alpha(s)$. Using now the connection \eqref{Eq_SL_T2} between the $S$-resolvents of $T$ and $T^2$, we can substitute $\xi=s^2$ to rewrite this integral as
\begin{align*}
(eq_\alpha)(T)&=\frac{1}{\pi}\int_{\partial S_\varphi\cap\mathbb{C}_J}S_L^{-1}(s^2,T^2)sds_J\frac{s^{2n}s^\alpha}{(1+s^2)^{2n}} \\
&=\frac{1}{2\pi}\int_{\partial S_{2\varphi}\cap\mathbb{C}_J}S_L^{-1}(\xi,T^2)d\xi_J\frac{\xi^n\xi^{\frac{\alpha}{2}}}{(1+\xi)^{2n}}=\big(\widetilde{e}\,\xi^{\frac{\alpha}{2}}\big)(T^2),
\end{align*}
where in the last equation we interpreted the integral as the sectorial $\omega$-functional calculus \cite[Definition~7.1.5]{FJBOOK} with the regularizer function $\widetilde{e}(\xi):=\frac{\xi^n}{(1+\xi)^{2n}}$ and the fractional power $\xi^{\frac{\alpha}{2}}$. \medskip

Using now the $H^\infty$-functional calculus \eqref{Eq_Hinfty}, we get the stated result
\begin{equation*}
q_\alpha(T)=e(T)^{-1}(eq_\alpha)(T)=\widetilde{e}(T^2)^{-1}\big(\widetilde{e}\,\xi^{\frac{\alpha}{2}}\big)(T^2)=(\xi^{\frac{\alpha}{2}})(T^2)=(T^2)^{\frac{\alpha}{2}}.
\end{equation*}
The second identity $p_\alpha(T)=\overline{\sgn(T)(T^2)^{\frac{\alpha}{2}}}$ now follows immediately from the power rule Theorem~\ref{thm_Product_rule_mixed_powers}~i) and the obvious fact that $p_0(s)=\sgn(s)$.
\end{proof}

\section{The gradient with nonconstant coefficients}\label{sec_Gradient}

In this section we want to apply the abstract results of the previous Section~\ref{sec_Fractional_powers} to the specific case of the gradient operator with non-constant coefficients in $n\geq 3$ dimensions. Precisely, if $e_1,\dots,e_n$ are the imaginary units of the Clifford algebra $\mathbb{R}_n$, we consider the hypercomplex representation of the gradient
\begin{equation}\label{Eq_Gradient}
\nabla_a:=\sum_{i=1}^ne_ia_i(x)\frac{\partial}{\partial x_i},\qquad\text{with }\dom(\nabla_a):=H^1(\mathbb{R}^n).
\end{equation}
This is an operator in the Clifford module of square integrable functions with values in $\mathbb{R}_n$,
\begin{equation*}
L^2(\mathbb{R}^n)=\Big\{f:\mathbb{R}^n\to\mathbb{R}_n\;\Big|\;\int_{\mathbb{R}^n}|f(x)|^2dx<\infty\Big\}.
\end{equation*}
Our goal is to define the $H^\infty$-functional calculus of $\nabla_a$, and in particular apply the fractional powers of Section~\ref{sec_Fractional_powers}. \medskip

For the coefficients $a_1,\dots,a_n$ of $\nabla_a$ in \eqref{Eq_Gradient}, we require the following assumptions.

\begin{ass}\label{ass_Coefficients}
The coefficients $a_1,\dots,a_n:\mathbb{R}^n\rightarrow\mathbb{R}$ of the gradient \eqref{Eq_Gradient} are continuously differentiable and belong to one of following two cases:

\begin{itemize}
\item[I)]  If the functions $a_i$ depend on all variables $x_1,\dots,x_n$, we assume that
\begin{subequations}
\begin{align}
m_a:=&\min\limits_{i\in\{1,\dots,n\}}\inf\limits_{x\in\mathbb{R}^n}a_i(x)>0, \label{Eq_ma} \\
M_a:=&\Big(\sum\nolimits_{i=1}^n\Vert a_i\Vert_{L^\infty}^2\Big)^{\frac{1}{2}}<\infty, \label{Eq_Ma} \\
M_a':=&\Big(\sum\nolimits_{i,j=1}^n\Big\Vert a_j\frac{\partial a_i}{\partial x_j}\Big\Vert^2_{L^n}\Big)^{\frac{1}{2}}<\infty, \label{Eq_Maprime} \\
M_a'':=&\Big(\sum\nolimits_{i,j=1}^n\Big\Vert\frac{\partial a_i}{\partial x_j}\Big\Vert_{L^\infty}^2\Big)^{\frac{1}{2}}<\infty, \label{Eq_Maprimeprime}
\end{align}
\end{subequations}
with bounds which satisfy the condition
\begin{equation}\label{Eq_Coefficient_inequality}
m_a^2>C_SM_a',
\end{equation}
with $$C_S=\frac{1}{\sqrt{\pi n(n-2)}}\Big(\frac{\Gamma(n)}{\Gamma(\frac{n}{2})}\Big)^{\frac{1}{n}}$$ the constant from Sobolev embedding theorem, see \cite{T76}.

\item[II)] If $a_i(x)=a_i(x_i)$ only depends on the variable $x_i$, for every $i\in\{1,\dots,n\}$, we only assume the bounds \eqref{Eq_ma}, \eqref{Eq_Ma}, and \eqref{Eq_Maprimeprime}.
\end{itemize}
\end{ass}

\begin{rem}
In case I) of Assumption~\ref{ass_Coefficients}, the conditions \eqref{Eq_ma}, \eqref{Eq_Ma} and \eqref{Eq_Maprime} are needed for the existence of a weak solution of the spectral problem, while the condition \eqref{Eq_Maprimeprime} leads to $H^2$-regularity of this solution. The inequality \eqref{Eq_Coefficient_inequality} additionally ensures that the $S$-spectrum is contained in a double sector. Case II) on the other hand will be reduced to a gradient with constant coefficients using a transformation of variables.
\end{rem}

In order to derive the $S$-spectrum of the operator $\nabla_a$, we need to investigate the bounded invertibility of the operator $Q_s[\nabla_a]$ in \eqref{Eq_Qs}. As calculated in \cite[Eq.(1.7)]{Gradient}, the corresponding form to this operator is given by
\begin{equation}\label{Eq_qs}
q_s(u,v):=\sum_{i=1}^n\Big\langle\frac{\partial u}{\partial x_i},a_i^2\frac{\partial v}{\partial x_i}+(2s_0a_i-B_i)e_iv\Big\rangle_{L^2}+|s|^2\langle u,v\rangle_{L^2},
\end{equation}
with $\dom(q_s):=H^1(\mathbb{R}^n)$, and where we have set
\begin{equation}\label{Eq_Bi}
B_i:=\sum\limits_{j=1}^ne_ja_j\frac{\partial a_i}{\partial x_j},\qquad i\in\{1,\dots,n\}.
\end{equation}
Under the Assumption~\ref{ass_Coefficients} on the coefficients, we now prove the unique solvability of the weak formulation.

\begin{thm}\label{thm_Weak_solution}
Let the coefficients $a_1,\dots,a_n$ satisfy the Assumption~\ref{ass_Coefficients} case I) or case II), and let us set
\begin{equation}\label{Eq_Ka}
K_a:=\begin{cases} \frac{M_a}{\sqrt{m_a^2-C_SM_a'}}, & \text{in case I)}, \\ 1, & \text{in  case II)}. \end{cases}
\end{equation}
Then, for every $s\in\mathbb{R}^{n+1}$ with
\begin{equation*}
|s|>K_a|s_0|,
\end{equation*}
and every $f\in L^2(\mathbb{R}^n)$, there exists a unique $u_f\in H^1(\mathbb{R}^n)$, such that
\begin{equation*}
q_s(u_f,v)=\langle f,v\rangle_{L^2},\qquad v\in H^1(\mathbb{R}^n).
\end{equation*}
Moreover, with the notion $$\Vert u\Vert_D^2:=\sum_{j=1}^n\Vert\frac{\partial u}{\partial x_j}\Vert_{L^2}^2,$$ this solution satisfies the bounds
\begin{equation}\label{Eq_Estimate_weak_solution}
\Vert u_f\Vert_{L^2}\leq\Big(\frac{M_a}{m_a\sqrt{n}}\Big)^{\frac{n}{2}}\frac{\Vert f\Vert_{L^2}}{|s|^2-K_a^2s_0^2},\quad\text{and}\quad\Vert u_f\Vert_D\leq\Big(\frac{M_a}{m_a\sqrt{n}}\Big)^{\frac{n}{2}}\frac{K_a|s|\Vert f\Vert_{L^2}}{m_a(|s|^2-K_a^2s_0^2)}.
\end{equation}
\end{thm}

\begin{proof}
In the case where the coefficients $a_i$ satisfy the Assumption~\ref{ass_Coefficients} case~I), this result is already proven in \cite[Theorem~3.2]{Gradient}. Note that there is always $M_a\geq\sqrt{n}\,m_a$ for the constants in \eqref{Eq_ma} and \eqref{Eq_Ma}. In the case the coefficients satisfy the Assumption~\ref{ass_Coefficients} case~II), let us define for every $i\in\{1,\dots,n\}$ the function
\begin{equation*}
\xi_i(x_i):=\int_0^{x_i}\frac{1}{a_i(u)}du,\qquad x_i\in\mathbb{R}.
\end{equation*}
Since $\frac{1}{a_i(u)}\leq\frac{1}{m_a}$ is bounded from above by \eqref{Eq_ma}, the integral exists. Since $\frac{1}{a_i(u)}\geq\frac{1}{M_a}$ is also bounded from below by \eqref{Eq_Ma}, the function $\xi_i$ is strictly monotone increasing, with $\lim_{x_i\rightarrow\pm\infty}\xi_i(x_i)=\pm\infty$. Hence $\xi_i:\mathbb{R}\rightarrow\mathbb{R}$ is bijective, and differentiable with derivative
\begin{equation*}
\xi_i'(x_i)=\frac{1}{a_i(x_i)}.
\end{equation*}
We now use the change of variables
\begin{equation}\label{Eq_Transformation_of_variables}
y_i=\xi_i(x_i),\qquad i\in\{1,\dots,n\},
\end{equation}
to reduce the gradient $\nabla_a$ with nonconstant coefficients, to the gradient with all coefficients equal to one
\begin{equation*}
\nabla_a=\sum_{i=1}^ne_ia_i(x_i)\frac{\partial}{\partial x_i}=\sum_{i=1}^ne_ia_i(x_i)\xi_i'(x_i)\frac{\partial}{\partial y_i}=\sum_{i=1}^ne_i\frac{\partial}{\partial y_i}=:\nabla.
\end{equation*}
The form $\widetilde{q}_s$, which corresponds to the operator $Q_s[\nabla]$, is given by \eqref{Eq_qs} when all coefficients $a_i$ are set to $1$, i.e.
\begin{equation*}
\widetilde{q}_s(u,v):=\sum_{i=1}^n\Big\langle\frac{\partial u}{\partial y_i},\frac{\partial v}{\partial y_i}\Big\rangle_{L^2}+|s|^2\langle u,v\rangle_{L^2},\qquad u,v\in H^1(\mathbb{R}^n).
\end{equation*}
For functions $u,v\in H^1(\mathbb{R}^n)$ let us now consider the functions
\begin{equation}\label{Eq_utilde}
\widetilde{u}(y):=u(x),\qquad\text{and}\qquad\widehat{v}(y):=a_1(x_1)\dots a_n(x_n)v(x),
\end{equation}
with the connection \eqref{Eq_Transformation_of_variables} between the coordinates $y$ and $x$. Then there is
\begin{equation}\label{Eq_u_utilde_connection}
u\in H^1(\mathbb{R}^n)\Leftrightarrow\widetilde{u}\in H^1(\mathbb{R}^n)\qquad\text{as well as}\qquad v\in H^1(\mathbb{R}^n)\Leftrightarrow\widetilde{v}\in H^1(\mathbb{R}^n).
\end{equation}
Moreover, the partial derivatives of these functions are given by
\begin{align*}
\frac{\partial}{\partial y_i}\widetilde{u}(y)&=a_i(x_i)\frac{\partial}{\partial x_i}u(x), \\
\frac{\partial}{\partial y_i}\widehat{v}(y)&=a_i(x_i)\frac{\partial}{\partial x_i}\big(a_1(x_1)\dots a_n(x_n)v(x)\big)=a_1(x_1)\dots a_n(x_n)\frac{\partial}{\partial x_i}\big(a_i(x_i)v(x)\big).
\end{align*}
With this notion of $\widetilde{u}$ and $\widehat{v}$ there exists the following connection between the forms $q_s$ and $\widetilde{q}_s$
\begin{align*}
q_s(u,v)&=\sum_{i=1}^n\Big\langle\frac{\partial u}{\partial x_i},a_i\Big(\frac{\partial}{\partial x_i}(va_i)+2s_0e_iv\Big)\Big\rangle_{L^2,dx}+|s|^2\langle u,v\rangle_{L^2,dx} \\
&=\sum_{i=1}^n\Big\langle\frac{1}{a_i}\frac{\partial\widetilde{u}}{\partial y_i},\frac{a_i}{a_1\dots a_n}\Big(\frac{\partial\widehat{v}}{\partial y_i}+2s_0e_i\widehat{v}\Big)\Big\rangle_{L^2,dx}+|s|^2\Big\langle\widetilde{u},\frac{\widehat{v}}{a_1\dots a_n}\Big\rangle_{L^2,dx} \\
&=\sum_{i=1}^n\Big\langle\frac{\partial\widetilde{u}}{\partial y_i},\frac{\partial\widehat{v}}{\partial y_i}+2s_0e_i\widehat{v}\Big\rangle_{L^2,dy}+|s|^2\langle\widetilde{u},\widehat{v}\rangle_{L^2,dy}=\widetilde{q}_s(\widetilde{u},\widehat{v}).
\end{align*}

Note that in the index of the $L^2$-scalar products we indicate by $dx$ or $dy$ with respect to what variable the $L^2$-integral is understood. This connection between the forms means that for $f\in L^2(\mathbb{R}^n)$, the problem
\begin{equation}\label{Eq_qs_qstilde_equivalence_1}
q_s(u,v)=\langle f,v\rangle_{L^2,dx},\qquad v\in H^1(\mathbb{R}^n),
\end{equation}
admits a unique weak solution $u_f\in H^1(\mathbb{R}^n)$ if and only if the problem
\begin{equation}\label{Eq_qs_qstilde_equivalence_2}
\widetilde{q}_s(\widetilde{u},\widehat{v})=\langle\widetilde{f},\widehat{v}\rangle_{L^2,dy},\qquad\widehat{v}\in H^1(\mathbb{R}^n),
\end{equation}
admits a weak solution $\widetilde{u}_{\widetilde{f}}\in H^1(\mathbb{R}^n)$ with respect to the function
\begin{equation}\label{Eq_ftilde}
\widetilde{f}(y):=f(x).
\end{equation}
However, since the constant coefficients of the form $\widetilde{q}_s$ are covered by Assumption~\ref{ass_Coefficients} case~I), there exists for every $\widetilde{f}\in L^2(\mathbb{R}^n)$ a unique solution $\widetilde{u}_{\widetilde{f}}$, which satisfies the estimate
\begin{equation*}
\Vert\widetilde{u}_{\widetilde{f}}\Vert_{L^2,dy}\leq\frac{\Vert\widetilde{f}\Vert_{L^2,dy}}{|\Im(s)|^2},\qquad\text{and}\qquad\Vert\widetilde{u}_{\widetilde{f}}\Vert_{D,dy}\leq\frac{|s|\Vert\widetilde{f}\Vert_{L^2,dy}}{|\Im(s)|^2}.
\end{equation*}
Plugging in the connection $\widetilde{u}_{\widetilde{f}}(y)=u_f(x)$ from \eqref{Eq_utilde}, the respective connection \eqref{Eq_ftilde} between $f$ and $\widetilde{f}$ as well as using the transformation $dx=a_1\dots a_ndy$ between the $dy$- and the $dx$-norms, gives
\begin{align*}
\Vert u_f\Vert_{L^2,dx}&=\big\Vert\sqrt{a_1\dots a_n}\,\widetilde{u}_{\widetilde{f}}\big\Vert_{L^2,dy}\leq\Big(\frac{M_a}{\sqrt{n}}\Big)^{\frac{n}{2}}\Vert\widetilde{u}_{\widetilde{f}}\Vert_{L^2,dy}\leq\Big(\frac{M_a}{\sqrt{n}}\Big)^{\frac{n}{2}}\frac{\Vert\widetilde{f}\Vert_{L^2,dy}}{|\Im(s)|^2} \\
&=\Big(\frac{M_a}{\sqrt{n}}\Big)^{\frac{n}{2}}\frac{1}{|\Im(s)|^2}\Big\Vert\frac{f}{\sqrt{a_1\dots a_n}}\Big\Vert_{L^2,dx}\leq\Big(\frac{M_a}{m_a\sqrt{n}}\Big)^{\frac{n}{2}}\frac{\Vert f\Vert_{L^2,dx}}{|\Im(s)|^2},
\end{align*}
where in the first line we used the inequality of arithmetic and geometric means, to estimate the term $\sqrt{a_1\dots a_n}$. Similarly, for the derivatives, the estimate on $\widetilde{u}_{\widetilde{f}}$ transforms into the estimate for $u_f$ as
\begin{equation*}
\Vert u_f\Vert_{D,dx}\leq\Big(\frac{M_a}{m_a\sqrt{n}}\Big)^{\frac{n}{2}}\frac{|s|\Vert f\Vert_{L^2,dx}}{m_a|\Im(s)|^2}. \qedhere
\end{equation*}
\end{proof}

Next we prove that the unique weak solution of the form $q_s$ from Theorem~\ref{thm_Weak_solution} is actually a strong solution of the operator $Q_s[\nabla_a]$. The proof of this statement is highly motivated by \cite[Theorem~8.8]{GT01}.

\begin{thm}\label{thm_Strong_solution}
Let the coefficients $a_1,\dots,a_n$ satisfy Assumption~\ref{ass_Coefficients}. If for some $s\in\mathbb{R}^{n+1}$ and $f\in L^2(\mathbb{R}^n)$, the function $u\in H^1(\mathbb{R}^n)$ is a solution of the weak problem
\begin{equation}\label{Eq_Strong_solution}
q_s(u,v)=\langle f,v\rangle_{L^2},\qquad v\in H^1(\mathbb{R}^n),
\end{equation}
with the form $q_s$ in \eqref{Eq_qs}, then $u\in H^2(\mathbb{R}^n)$ is a solution of
\begin{equation*}
Q_s[\nabla_a]u=f.
\end{equation*}
\end{thm}

\begin{proof}
In order to show that $u\in H^2(\mathbb{R}^n)$, we have to show that $\frac{\partial u}{\partial x_i}\in H^1(\mathbb{R}^n)$, for every $i\in\{1,\dots,n\}$. However, by \cite[Lemma~7.24]{GT01} it is sufficient to verify
\begin{equation}\label{Eq_H2_regularity_4}
\sup\limits_{h\in\mathbb{R}^n\setminus\{0\}}\Big\Vert\Delta_h\frac{\partial u}{\partial x_i}\Big\Vert_{L^2}<\infty,\qquad i\in\{1,\dots,n\},
\end{equation}
where for every $h\in\mathbb{R}^n\setminus\{0\}$, the difference quotient operator $\Delta_h:L^2(\mathbb{R}^n)\rightarrow L^2(\mathbb{R}^n)$ is defined as
\begin{equation*}
\Delta_hv(x):=\frac{v(x+h)-v(x)}{|h|},\qquad v\in L^2(\mathbb{R}^n),\ \ x\in\mathbb{R}^n.
\end{equation*}
Since $u$ is a solution of \eqref{Eq_Strong_solution}, we can in particular test \eqref{Eq_Strong_solution} with $\Delta_{-h}v$, for some arbitrary $v\in H^1(\mathbb{R}^n)$. Plugging in the explicit form \eqref{Eq_qs} of $q_s$, then gives
\begin{align*}
\langle f,\Delta_{-h}v\rangle_{L^2}&=\sum\limits_{i=1}^n\Big\langle\frac{\partial u}{\partial x_i},a_i^2\frac{\partial\Delta_{-h}v}{\partial x_i}+(2s_0a_i-B_i)e_i\Delta_{-h}v\Big\rangle_{L^2}+|s|^2\langle u,\Delta_{-h}v\rangle_{L^2} \\
&=\sum\limits_{i=1}^n\Big\langle a_i^2\frac{\partial u}{\partial x_i},\frac{\partial\Delta_{-h}v}{\partial x_i}\Big\rangle_{L^2}+\Big\langle|s|^2u-\sum\limits_{i=1}^ne_i(2s_0a_i+B_i)\frac{\partial u}{\partial x_i},\Delta_{-h}v\Big\rangle_{L^2},
\end{align*}
where in the second line we used $\overline{B_i}=-B_i$, see \eqref{Eq_Bi}. This equation can be rearranged to
\begin{equation}\label{Eq_H2_regularity_5}
\Big\langle\underbrace{f-|s|^2u+\sum\limits_{i=1}^ne_i(2s_0a_i+B_i)\frac{\partial u}{\partial x_i}}_{=:g},\Delta_{-h}v\Big\rangle_{L^2}=\sum\limits_{i=1}^n\Big\langle a_i^2\frac{\partial u}{\partial x_i},\frac{\partial\Delta_{-h}v}{\partial x_i}\Big\rangle_{L^2}.
\end{equation}
It is straight forward to verify that the adjoint difference quotient is given by $\Delta_{-h}^*=-\Delta_h$. Using also the product rule
\begin{equation*}
\Delta_h(uv)(x)=u(x+h)\Delta_hv(x)+\Delta_hu(x)v(x),
\end{equation*}
allows us to rewrite \eqref{Eq_H2_regularity_5} as
\begin{equation*}
\langle g,\Delta_{-h}v\rangle_{L^2}=-\sum\limits_{i=1}^n\Big\langle\Delta_h\Big(a_i^2\frac{\partial u}{\partial x_i}\Big),\frac{\partial v}{\partial x_i}\Big\rangle_{L^2}=-\sum\limits_{i=1}^n\Big\langle a_i(\,\cdot\,+h)^2\frac{\partial\Delta_hu}{\partial x_i}+\Delta_ha_i^2\frac{\partial u}{\partial x_i},\frac{\partial v}{\partial x_i}\Big\rangle_{L^2}.
\end{equation*}
Rewriting this equation once more, and using the inequalities
\begin{align*}
\Vert\Delta_{-h}v\Vert_{L^2}&\leq\frac{1}{|h|}\Big\Vert\sum\limits_{j=1}^nh_j\frac{\partial v}{\partial x_j}\Big\Vert_{L^2}\leq\Vert v\Vert_D,\qquad\text{and} \\
\Vert\Delta_ha_i^2\Vert_{L^\infty}&\leq\frac{1}{|h|}\Big\Vert\sum\limits_{j=1}^nh_j\frac{\partial a_i^2}{\partial x_j}\Big\Vert_{L^\infty}\leq\frac{2}{|h|}\Big\Vert\sum\limits_{j=1}^nh_ja_i\frac{\partial a_i}{\partial x_j}\Big\Vert_{L^\infty}\leq 2M_aM_a'',
\end{align*}
from \cite[Lemma~7.23]{GT01}, we can now estimate
\begin{align}
\Big|\sum\limits_{i=1}^n\Big\langle a_i(\,\cdot\,+h)^2\frac{\partial\Delta_hu}{\partial x_i},\frac{\partial v}{\partial x_i}\Big\rangle_{L^2}\Big|&\leq\sum\limits_{i=1}^n\Vert\Delta_ha_i^2\Vert_{L^\infty}\Big\Vert\frac{\partial u}{\partial x_i}\Big\Vert_{L^2}\Big\Vert\frac{\partial v}{\partial x_i}\Big\Vert_{L^2}+\Vert g\Vert_{L^2}\Vert\Delta_{-h}v\Vert_{L^2} \notag \\
&\leq 2M_aM_a''\sum\limits_{i=1}^n\Big\Vert\frac{\partial u}{\partial x_i}\Big\Vert_{L^2}\Big\Vert\frac{\partial v}{\partial x_i}\Big\Vert_{L^2}+\Vert g\Vert_{L^2}\Vert\Delta_{-h}v\Vert_{L^2} \notag \\
&\leq 2M_aM_a''\Vert u\Vert_D\Vert v\Vert_D+\Vert g\Vert_{L^2}\Vert v\Vert_D.\label{Eq_H2_regularity_1}
\end{align}
Next, we will further estimate the $L^2$-norm of the function $g$ in \eqref{Eq_H2_regularity_5}, namely
\begin{align}
\Vert g\Vert_{L^2}&\leq\Vert f\Vert_{L^2}+|s|^2\Vert u\Vert_{L^2}+\sum\limits_{i=1}^n\Big(2|s_0|\Vert a_i\Vert_{L^\infty}+\Vert B_i\Vert_{L^\infty}\Big)\Big\Vert\frac{\partial u}{\partial x_i}\Big\Vert_{L^2} \notag \\
&\leq\Vert f\Vert_{L^2}+|s|^2\Vert u\Vert_{L^2}+\bigg(2|s_0|\Big(\sum\limits_{i=1}^n\Vert a_i\Vert_{L^\infty}^2\Big)^{\frac{1}{2}}+\Big(\sum\limits_{i=1}^n\Vert B_i\Vert_{L^\infty}^2\Big)^{\frac{1}{2}}\bigg)\Big(\sum\limits_{i=1}^n\Big\Vert\frac{\partial u}{\partial x_i}\Big\Vert_{L^2}^2\Big)^{\frac{1}{2}} \notag \\
&\leq\Vert f\Vert_{L^2}+|s|^2\Vert u\Vert_{L^2}+(2|s_0|M_a+M_aM_a'')\Vert u\Vert_D, \label{Eq_H2_regularity_3}
\end{align}
where in the last line we used that the bounds \eqref{Eq_Ma} and \eqref{Eq_Maprimeprime}, to estimate the $B_i$-term, given in \eqref{Eq_Bi}, by
\begin{equation*}
\Vert B_i\Vert_{L^\infty}\leq\sum_{j=1}^n\Vert a_j\Vert_{L^\infty}\Big\Vert\frac{\partial a_i}{\partial x_j}\Big\Vert_{L^\infty}\leq\Big(\sum_{j=1}^n\Vert a_j\Vert_{L^\infty}^2\Big)^{\frac{1}{2}}\Big(\sum\limits_{j=1}^n\Big\Vert\frac{\partial a_i}{\partial x_j}\Big\Vert_{L^\infty}^2\Big)^{\frac{1}{2}}=M_a\Big(\sum\limits_{j=1}^n\Big\Vert\frac{\partial a_i}{\partial x_j}\Big\Vert_{L^\infty}^2\Big)^{\frac{1}{2}}.
\end{equation*}
Plugging now \eqref{Eq_H2_regularity_3} back into \eqref{Eq_H2_regularity_1}, gives
\begin{equation*}
\Big|\sum\limits_{i=1}^n\Big\langle a_i(\,\cdot\,+h)^2\frac{\partial\Delta_hu}{\partial x_i},\frac{\partial v}{\partial x_i}\Big\rangle_{L^2}\Big|\leq\Big(\Vert f\Vert_{L^2}+|s|^2\Vert u\Vert_{L^2}+(3M_a''+2|s_0|)M_a\Vert u\Vert_D\Big)\Vert v\Vert_D.
\end{equation*}
With the special choice $v=\Delta_hu$, this inequality then becomes
\begin{equation*}
\sum\limits_{i=1}^n\Big\langle a_i(\,\cdot\,+h)^2\frac{\partial\Delta_hu}{\partial x_i},\frac{\partial\Delta_hu}{\partial x_i}\Big\rangle_{L^2}\leq\Big(\Vert f\Vert_{L^2}+|s|^2\Vert u\Vert_{L^2}+(3M_a''+2|s_0|)M_a\Vert u\Vert_D\Big)\Vert\Delta_hu\Vert_D.
\end{equation*}
Note that we were allowed to remove the absolute value on the left hand side, since every inner product is nonnegative. Using now also the lower bound \eqref{Eq_ma} of the coefficients $a_i$, then gives
\begin{equation*}
m_a^2\Vert\Delta_hu\Vert_D^2\leq\Big(\Vert f\Vert_{L^2}+|s|^2\Vert u\Vert_{L^2}+(3M_a''+2|s_0|)M_a\Vert u\Vert_D\Big)\Vert\Delta_hu\Vert_D.
\end{equation*}
Simplifying $\Vert\Delta_hu\Vert_D$ on both sides of this inequality, we have the desired uniform boundedness \eqref{Eq_H2_regularity_4}. By \cite[Lemma~7.24]{GT01} this then implies $u\in H^2(\mathbb{R}^n)$. \medskip

Now that we know that $u\in H^2(\mathbb{R}^n)$, we can undo the integration by parts in \cite[Eq.(1.8)]{Gradient}, to rewrite the weak solution as
\begin{equation*}
\langle f,v\rangle_{L^2}=q_s(u,v)=\langle Q_s[\nabla_a]u,v\rangle_{L^2},\qquad v\in H^1(\mathbb{R}^n).
\end{equation*}
Since $v\in H^1(\mathbb{R}^n)$ is arbitrary, we conclude that $u$ is a strong solution of $Q_s[\nabla_a]u=f$.
\end{proof}

\begin{cor}
Let the coefficients $a_1,\dots,a_n$ satisfy the Assumption~\ref{ass_Coefficients}. Then
\begin{equation*}
\dom(\nabla_a^2)=H^2(\mathbb{R}^n),
\end{equation*}
and $\nabla_a^2$ acts as on any $u\in\dom(\nabla_a^2)$ as
\begin{equation}\label{Eq_Action_nabla_squared}
\nabla_a^2u=\sum\limits_{i,j=1}^ne_ie_ja_i\frac{\partial a_j}{\partial x_i}\frac{\partial u}{\partial x_j}-\sum\limits_{i=1}^na_i^2\frac{\partial^2u}{\partial x_i^2}.
\end{equation}
\end{cor}

\begin{proof}
It is clear that $H^2(\mathbb{R})\subseteq\dom(\nabla_a^2)$. For the inverse inclusion, let $u\in\dom(\nabla_a^2)$. Then with the form \eqref{Eq_qs} with $s=0$, we can write for every test function $v\in D(\mathbb{R}^n)$
\begin{align}
q_0(u,v)&=\sum_{i=1}^n\Big\langle\frac{\partial u}{\partial x_i},a_i^2\frac{\partial v}{\partial x_i}-B_ie_iv\Big\rangle_{L^2} \notag \\
&=\sum_{i=1}^n\Big\langle a_i\frac{\partial u}{\partial x_i},\frac{\partial}{\partial x_i}(a_iv)\Big\rangle_{L^2}-\sum_{i=1}^n\Big\langle\frac{\partial u}{\partial x_i},\Big(B_i-e_ia_i\frac{\partial a_i}{\partial x_i}\Big)e_iv\Big\rangle_{L^2} \notag \\
&=-\sum_{i=1}^n\Big\langle a_i\frac{\partial}{\partial x_i}\Big(a_i\frac{\partial u}{\partial x_i}\Big),v\Big\rangle_{D'\times D}-\sum_{i=1}^n\Big\langle e_i\Big(B_i-e_ia_i\frac{\partial a_i}{\partial x_i}\Big)\frac{\partial u}{\partial x_i},v\Big\rangle_{D'\times D}, \label{Eq_dom_nabla2_1}
\end{align}
where in the second last line we used integration by parts in the sense of distributions, since we a priori do not know if $a_i\frac{\partial}{\partial x_i}(a_i\frac{\partial u}{\partial x_i})\in L^2(\mathbb{R}^n)$. Using now also the explicit representation
\begin{equation*}
\nabla_a^2u=-\sum_{i=1}^na_i\frac{\partial}{\partial x_i}\Big(a_i\frac{\partial u}{\partial x_i}\Big)-\sum_{i=1}^ne_i\Big(B_i-e_ia_i\frac{\partial a_i}{\partial x_i}\Big)\frac{\partial u}{\partial x_i}\in L^2(\mathbb{R}^n),
\end{equation*}
from \cite[Eq.(1.6)]{CMS24}, we can rewrite \eqref{Eq_dom_nabla2_1} as
\begin{equation*}
q_0(u,v)=\langle\nabla_a^2u,v\rangle_{D'\times D}=\langle\nabla_a^2u,v\rangle_{L^2},\qquad v\in D(\mathbb{R}^n),
\end{equation*}
Note that in the second equation we again replaced the $D'\times D$-dual pairing by the $L^2$-inner product, since we know that $\nabla_a^2u\in L^2(\mathbb{R}^n)$. However, since $D(\mathbb{R}^n)$ is dense in $H^1(\mathbb{R}^n)$, we can extend this identity to
\begin{equation*}
q_0(u,v)=\langle\nabla_a^2u,v\rangle_{L^2},\qquad v\in H^1(\mathbb{R}^n).
\end{equation*}
This means, the function $u$ is a weak solution of \eqref{Eq_Strong_solution}, with the function $f=\nabla^2_au\in L^2(\mathbb{R}^n)$. Consequently, Theorem~\ref{thm_Strong_solution} then states $u\in H^2(\mathbb{R}^n)$. \medskip

In order to show the explicit action of $\nabla_a^2$ in \eqref{Eq_Action_nabla_squared}, let $u\in\dom(\nabla_a^2)=H^2(\mathbb{R}^n)$. Applying \eqref{Eq_Gradient} twice onto this function, and using the product rule gives
\begin{align*}
\nabla_a^2u&=\sum\limits_{i=1}^ne_ia_i\frac{\partial}{\partial x_i}\sum\limits_{j=1}^ne_ja_j\frac{\partial u}{\partial x_j}=\sum\limits_{i,j=1}^ne_ie_ja_i\Big(\frac{\partial a_j}{\partial x_i}\frac{\partial u}{\partial x_j}+a_j\frac{\partial^2u}{\partial x_i\partial x_j}\Big) \\
&=\sum\limits_{i,j=1}^ne_ie_ja_i\frac{\partial a_j}{\partial x_i}\frac{\partial u}{\partial x_j}-\sum\limits_{i=1}^na_i^2\frac{\partial^2u}{\partial x_i^2},
\end{align*}
where in the last line we used that all summands with $i\neq j$ cancel because of the anticommutativity $e_ie_j=-e_je_i$ of the Clifford imaginary units.
\end{proof}

The last main result which is missing to apply the $H^\infty$-functional calculus to the gradient, is the injectivity.

\begin{prop}\label{prop_Gradient_injective}
Let the coefficients $a_1,\dots,a_n$ satisfy the Assumption~\ref{ass_Coefficients}. Then, the gradient operator $\nabla_a$ in \eqref{Eq_Gradient} is injective.
\end{prop}

\begin{proof}
We distinguish whether the coefficients $a_1,\dots,a_n$ satisfy Assumption~\ref{ass_Coefficients}~I) or II). \medskip

In case I), let $u\in H^1(\mathbb{R}^n)$ with $\nabla_au=0$. Then clearly $u\in\dom(\nabla_a^2)$, with $Q_0[\nabla_a]u=0$. Hence $u$ must also be a weak solution of the equation
\begin{equation*}
q_0(u,v)=0,\qquad v\in H^1(\mathbb{R}^n).
\end{equation*}
In particular, for the special choice $v=u$, we have
\begin{equation}\label{Eq_Injectivity_1}
q_0(u,u)=0.
\end{equation}
However, from equation \cite[Eq.(3.7)]{Gradient} we know that the form $q_0$ is elliptic, i.e.
\begin{equation}\label{Eq_Injectivity_2}
\Sc q_0(u,u)\geq(m_a^2-C_SM_a')\Vert u\Vert_D^2.
\end{equation}
Combining now \eqref{Eq_Injectivity_1} and \eqref{Eq_Injectivity_2}, gives us $\Vert u\Vert_D=0$ and consequently $u$ is constant, which on $\mathbb{R}^n$ can only be the case if $u\equiv 0$. \medskip

In case II), let again $u\in H^1(\mathbb{R}^n)$ with $\nabla_au=0$. Using the transformation of variables \eqref{Eq_Transformation_of_variables}, as well as the function $\widetilde{u}(y)=u(x)$ from \eqref{Eq_utilde}, we can write
\begin{equation*}
0=\nabla_au(x)=\sum\limits_{i=1}^ne_ia_i(x_i)\frac{\partial}{\partial x_i}u(x)=\sum\limits_{i=1}^ne_i\frac{\partial}{\partial y_i}\widetilde{u}(y)=\nabla\widetilde{u}(y).
\end{equation*}
Since we know that $\nabla$ is injective by case I), we conclude that $\widetilde{u}(y)=0$, for every $y\in\mathbb{R}^n$ and hence also $u(x)=0$, for every $x\in\mathbb{R}^n$.
\end{proof}

Now we are in the position to combine the above results and prove that the gradient $\nabla_a$ is an injective bisectorial operator, and hence we are allowed to apply the $H^\infty$-functional calculus to it.

\begin{thm}\label{thm_Gradient_bisectorial}
Let the coefficients $a_1,\dots,a_n$ satisfy Assumption~\ref{ass_Coefficients}, and set $K_a$ as in \eqref{Eq_Ka}. Then the gradient operator $\nabla_a$ in \eqref{Eq_Gradient} is densely defined, injective and bisectorial of any angle
\begin{equation}\label{Eq_omega_range}
\arctan\big(\sqrt{K_a^2-1}\big)<\omega<\frac{\pi}{2}.
\end{equation}
\end{thm}

\begin{proof}
The domain $H^1(\mathbb{R}^n)$ of $\nabla_a$ is clearly dense in $L^2(\mathbb{R}^n)$ and the injectivity is proven in Proposition~\ref{prop_Gradient_injective}. In order to show that $\nabla_a$ is bisectorial, we have to check the Definition~\ref{defi_Bisectorial_operator}. To do so, let $s\in\mathbb{R}^{n+1}\setminus\overline{D_\omega}$, i.e. $|\Im(s)|>\tan(\omega)|\Sc(s)|$. By the choice of $\omega$ in \eqref{Eq_omega_range}, this implies
\begin{equation*}
|s|>K_a|s_0|.
\end{equation*}
Hence, by Theorem~\ref{thm_Weak_solution} and Theorem~\ref{thm_Strong_solution}; there exists for every $f\in L^2(\mathbb{R}^n)$ a unique solution $u_f\in\dom(\nabla_a^2)$, of
\begin{equation*}
Q_s[\nabla_a]u_f=f.
\end{equation*}
This shows that the operator $Q_s[\nabla_a]$ is bijective, and consequently, we have proven that
\begin{equation*}
\mathbb{R}^{n+1}\setminus\overline{D_\omega}\subseteq\rho_S(\nabla_a),
\end{equation*}
which is exactly \eqref{Eq_Bisectorial_Spectrum}. Moreover, by \eqref{Eq_Estimate_weak_solution}, the inverse operator satisfies the estimate
\begin{equation}\label{Eq_Gradient_bisectorial_1}
\Vert Q_s[\nabla_a]^{-1}f\Vert_{L^2}=\Vert u_f\Vert_{L^2}\leq\Big(\frac{M_a}{m_a\sqrt{n}}\Big)^{\frac{n}{2}}\frac{\Vert f\Vert_{L^2}}{|s|^2-K_a^2s_0^2}.
\end{equation}
Also by \eqref{Eq_Estimate_weak_solution}, the gradient of the pseudo $S$-resolvent admits the upper bound
\begin{align}
\Vert\nabla_aQ_s[\nabla_a]^{-1}f\Vert_{L^2}&\leq\sum\limits_{i=1}^n\Big\Vert a_i\frac{\partial u_f}{\partial x_i}\Big\Vert_{L^2}\leq\sum\limits_{i=1}^n\Vert a_i\Vert_{L^\infty}\Big\Vert\frac{\partial u_f}{\partial x_i}\Big\Vert_{L^2} \notag \\
&\leq\Big(\sum\limits_{i=1}^n\Vert a_i\Vert_{L^\infty}^2\Big)^{\frac{1}{2}}\Big(\sum\limits_{i=1}^n\Big\Vert\frac{\partial u_f}{\partial x_i}\Big\Vert_{L^2}^2\Big)^{\frac{1}{2}} \notag \\
&\leq M_a\Big(\frac{M_a}{m_a\sqrt{n}}\Big)^{\frac{n}{2}}\frac{K_a|s|\Vert f\Vert_{L^2}}{m_a(|s|^2-K_a^2s_0^2)}. \label{Eq_Gradient_bisectorial_2}
\end{align}
Combining now \eqref{Eq_Gradient_bisectorial_1} and \eqref{Eq_Gradient_bisectorial_2}, we can estimate the the left $S$-resolvent \eqref{Eq_SL_SR} by
\begin{align}
\Vert S_L^{-1}(s,\nabla_a)\Vert&\leq\Vert Q_s[\nabla_a]^{-1}\overline{s}\Vert+\Vert\nabla_aQ_s[\nabla_a]^{-1}\Vert \notag \\
&\leq\Big(\frac{M_a}{m_a\sqrt{n}}\Big)^{\frac{n}{2}}\Big(1+\frac{M_aK_a}{m_a}\Big)\frac{|s|}{|s|^2-K_a^2s_0^2},\qquad s\in\mathbb{R}^{n+1}\setminus\overline{D_\omega}. \label{Eq_Gradient_bisectorial_3}
\end{align}
For the estimate \eqref{Eq_SL_estimate}, let us fix now $\varphi\in(\omega,\frac{\pi}{2})$. Then for every $s\in\mathbb{R}^{n+1}\setminus(D_\varphi\cup\{0\})$, there is $|\Im(s)|\geq\tan(\varphi)|s_0|$ and consequently
\begin{equation*}
|s|^2\geq\big(1+\tan^2(\varphi)\big)^2s_0^2.
\end{equation*}
Taking into account \eqref{Eq_Gradient_bisectorial_3}, we then have
\begin{align*}
\Vert S_L^{-1}(s,\nabla_a)\Vert&\leq\Big(\frac{M_a}{m_a\sqrt{n}}\Big)^{\frac{n}{2}}\Big(1+\frac{M_aK_a}{m_a}\Big)\frac{|s|}{|s|^2-K_a^2\frac{|s|^2}{1+\tan^2(\varphi)}} \\
&=\Big(\frac{M_a}{m_a\sqrt{n}}\Big)^{\frac{n}{2}}\Big(1+\frac{M_aK_a}{m_a}\Big)\frac{1+\tan^2(\varphi)}{1+\tan^2(\varphi)-K_a^2}\frac{1}{|s|}.
\end{align*}
Finally, let $J \in \mathbb{S}$. We have just established that $J \in \rho_S(\nabla_a)$ and, therefore, by \cite[Theorem 2.4]{ADJOINT}, the operator $\nabla_a - \mathcal{I}^R J$ admits a bounded inverse. In particular, $\nabla_a - \mathcal{I}^R J$ is closed and, since $\mathcal{I}^R J$ is a bounded operator, it follows that $\nabla_a$ is also a closed operator.
\end{proof}

Next we want to prove that for the gradient operator with coefficients which satisfy the case II) of Assumption~\ref{ass_Coefficients}, the $S$-spectrum is the whole real line.

\begin{prop}
Let the coefficients $a_1,\dots,a_n$ satisfy Assumption~\ref{ass_Coefficients} case II). Then
\begin{equation*}
\sigma_S(\nabla_a)=\mathbb{R}.
\end{equation*}
\end{prop}

\begin{proof}
With the connection \eqref{Eq_u_utilde_connection} between $u$ and $\widetilde{u}$ as well as between $v$ and $\widehat{v}$, as well as the equivalence \eqref{Eq_qs_qstilde_equivalence_1} and \eqref{Eq_qs_qstilde_equivalence_2} of the solvability of $q_s$ and $\widetilde{q}_s$, we conclude from Theorem~\ref{thm_Strong_solution}, that for every $f\in L^2(\mathbb{R}^n)$ there holds the equivalence
\begin{align*}
Q_s[\nabla_a]u=f\quad&\Leftrightarrow\quad q_s(u,v)=\langle f,v\rangle_{L^2},\quad\text{for all }v\in H^1(\mathbb{R}^n), \\
&\Leftrightarrow\quad\widetilde{q}_s(\widetilde{u},\widehat{v})=\langle\widetilde{f},\widetilde{v}\rangle_{L^2},\quad\text{for all }\widehat{v}\in H^1(\mathbb{R}^n), \\
&\Leftrightarrow\quad Q_s[\nabla]\widetilde{u}=\widetilde{f}.
\end{align*}
This equivalence proves that $Q_s[\nabla_a]$ is bijective if and only if $Q_s[\nabla]$ is bijective. Consequently, there is $\sigma_S(\nabla_a)=\sigma_S(\nabla)$, and since we already know from \cite[Theorem~9.1.1]{FJBOOK} that $\sigma_S(\nabla)=\mathbb{R}$, the statement of the proposition is proven.
\end{proof}

The previous Theorem~\ref{thm_Gradient_bisectorial} has shown that the gradient operator $\nabla_a$ is well suited to apply the $H^\infty$-functional calculus from \cite[Definition~5.3]{MS24}. In particular, the fractional powers $p_\alpha(\nabla_a)$ and $q_\alpha(\nabla_a)$ from Definition~\ref{defi_palphaT_positive} and Definition~\ref{defi_palphaT_negative} of the gradient are well defined for every $\alpha\in\mathbb{R}$. By its definition in \eqref{Eq_Gradient}, the gradient operator $\nabla_a$ consists of partial derivatives which are multiplied with the imaginary units. As a direct consequence of this definition, for every real valued $f\in H^1(\mathbb{R}^n)$, there is
\begin{equation*}
\nabla_af\text{ has values in }\mathbb{R}^n:=\big\{s_1e_1+\dots+s_ne_n\;\big|\;s_1,\dots,s_n\in\mathbb{R}\big\}.
\end{equation*}
The following theorem proves that the fractional powers $p_\alpha$ preserve this property and that the fractional powers $q_\alpha$ is doing the opposite, it gives back a real valued function. This property is for example crucial if we consider the heat equation \eqref{Eq_Fractional_heat_equation}. Namely, if $p_\alpha(\nabla_a)u$ is not a vector function, it is not possible to apply the divergence to it.

\begin{thm}\label{thm_palpha_vector_operator}
Let the coefficients $a_1,\dots,a_n$ satisfy Assumption~\ref{ass_Coefficients} case II), and $\alpha\in\mathbb{R}$ be arbitrary.

\begin{enumerate}
\item[i)] If $u\in\dom(p_\alpha(\nabla_a))$ is real valued, then $p_\alpha(\nabla_a)u$ has values in $\mathbb{R}^n$. \medskip

\item[ii)] If $u\in\dom(q_\alpha(\nabla_a))$ is real valued, then $q_\alpha(\nabla_a)u$ is real valued as well.
\end{enumerate}
\end{thm}

\begin{proof}
Let us start with $q_\alpha(\nabla_a)$ in ii). First, it is already shown in Theorem~\ref{thm_palpha_qalpha_T2}, that
\begin{equation*}
q_\alpha(\nabla_a)=(\nabla_a^2)^{\frac{\alpha}{2}},
\end{equation*}
where the right hand side is understood as the $H^\infty$-functional calculus of sectorial operators. Since we have assumed that the coefficients $a_1,\dots,a_n$ satisfy the Assumption~\ref{ass_Coefficients} case II), they only depend on $a_i(x)=a_i(x_i)$. In particular, in the explicit action \eqref{Eq_Action_nabla_squared} of $\nabla_a^2$ all the mixed derivatives with $i\neq j$ in the first sum vanish, and we get
\begin{equation}\label{Eq_Vector_operator_2}
\nabla_a^2=-\sum\limits_{i=1}^na_ia_i'\frac{\partial}{\partial x_i}-\sum\limits_{i=1}^na_i^2\frac{\partial^2}{\partial x_i^2}.
\end{equation}
Clearly, if we apply it on a real valued function, it gives back a real valued function. \medskip

In order to show the same property for $(\nabla_a^2)^{\frac{\alpha}{2}}$, let us choose the regularizer $$e(s)=\frac{s^k}{(1+s)^{2k}},$$ for some $k>\frac{|\alpha|}{2}$. Then the $H^\infty$-functional calculus is given by
\begin{equation}\label{Eq_Vector_operator_3}
q_\alpha(\nabla_a)=(\nabla_a^2)^{\frac{\alpha}{2}}=e(\nabla_a^2)^{-1}\big(es^{\frac{\alpha}{2}}\big)(\nabla_a^2).
\end{equation}
Starting with the part $\big(es^{\frac{\alpha}{2}}\big)(T)$, we can use the $\omega$-functional calculus of sectorial operators, to write for every $\varphi\in(0,\pi)$, and any $J\in\mathbb{S}$
\begin{equation*}
\big(es^{\frac{\alpha}{2}}\big)(\nabla_a^2)=\frac{1}{2\pi}\int_{\partial S_\varphi\cap\mathbb{C}_J}S_L^{-1}(s,\nabla_a^2)ds_Je(s)s^{\frac{\alpha}{2}}.
\end{equation*}
As in the proof of \cite[Proposition~3.7]{MS24}, only considering the path $\gamma_+$, we can rewrite this integral as
\begin{equation*}
\big(es^{\frac{\alpha}{2}}\big)(\nabla_a^2)=\frac{1}{\pi}\int_0^\infty\big(A(r)x(r)-B(r)y(r)\big)dr,
\end{equation*}
using the operators and the functions
\begin{align*}
A(r)&=\big(r\cos(\varphi)-\nabla_a^2\big)Q_{re^{J\varphi}}[\nabla_a^2]^{-1}, \\
B(t)&=r\sin(\varphi)Q_{re^{J\varphi}}[\nabla_a^2]^{-1}, \\
x(r)&=-\cos(\varphi)\Im\big((es^{\frac{\alpha}{2}})(re^{J\varphi})\big)-\sin(\varphi)\Sc\big((es^{\frac{\alpha}{2}})(re^{J\varphi})\big), \\
y(r)&=\sin(\varphi)\Im\big((es^{\frac{\alpha}{2}})(re^{J\varphi})\big)-\cos(\varphi)\Sc\big((es^{\frac{\alpha}{2}})(re^{J\varphi})\big).
\end{align*}
Plugging this representation into \eqref{Eq_Vector_operator_3} and using the explicit representation of the regularizer $e$, gives
\begin{equation}\label{Eq_Vector_operator_1}
q_\alpha(\nabla_a)=\frac{1}{\pi}(1+\nabla_a^2)^{2k}(\nabla_a^2)^{-k}\int_0^\infty\big(A(r)x(r)-B(r)y(r)\big)dr.
\end{equation}
Since $\nabla_a^2$ in \eqref{Eq_Vector_operator_2} maps real valued function to real valued functions, all the elements in this representation have the same property. This means, if we apply $q_\alpha(\nabla_a)$ onto some real valued $u\in\dom(q_\alpha(\nabla_a))$, then $q_\alpha(\nabla_a)u$ is again real valued. \medskip

Turning now our attention to the mapping properties of $p_\alpha(\nabla_a)$ in part i) of this theorem, we know from Proposition~\ref{prop_Gradient_injective} that $\nabla_a$ is injective. Moreover, since the space $L^2(\mathbb{R}^n)$ is reflexive, because it is a Hilbert space, we get from \cite[Theorem~3.3~ii)]{CMS25} that $\overline{\ran}(\nabla_a)=L^2(\mathbb{R}^n)$. Hence we are allowed to use the power rule for mixed powers in Theorem~\ref{thm_Product_rule_mixed_powers}~i) and the representation of $p_1(\nabla_a)$ in Remark~\ref{rem_Integer_powers}, to write
\begin{equation*}
p_\alpha(\nabla_a)=\overline{q_{\alpha-1}(\nabla_a)\nabla_a}.
\end{equation*}
Let now $u\in\dom(p_\alpha(\nabla_a))$ be real valued. Then there exists $(u_n)_n\in\dom\big(q_{\alpha-1}(\nabla_a)\nabla_a\big)$, such that
\begin{equation}\label{Eq_Vector_operator_4}
u=\lim\limits_{n\rightarrow\infty}u_n\qquad\text{and}\qquad p_\alpha(\nabla_a)u=\lim\limits_{n\rightarrow\infty}q_{\alpha-1}(\nabla_a)\nabla_au_n.
\end{equation}
Note that, according to Lemma~\ref{lem_Core}, the sequence $u_n$ is explicitly given by
\begin{equation*}
    u_n := n^{2m} \nabla_a^{2m} (\nabla_a^2+n^2)^{-m}\left(\nabla_a^2+\frac{1}{n^2}\right)^{-m}u
\end{equation*}
where $m$ is an integer such that $m \ge \frac{\alpha+1}{2}$; therefore, since it has already been established that $\nabla_a^2$ gives back a real-valued function whenever it is applied to a real-valued function and $u$ itself is real-valued, we conclude that each $u_n$ is real-valued as well. This means that $\nabla_au_n$ is a pure vector valued function, and using the representation \eqref{Eq_Vector_operator_1} of $q_{\alpha-1}(\nabla_a)$, one sees that also $q_{\alpha-1}(\nabla_a)\nabla_au_n$ is a vector valued function again. Consequently, also $p_\alpha(\nabla_a)u$ in the second limit in \eqref{Eq_Vector_operator_4} is a pure vector valued function.
\end{proof}

\section{Concluding remarks}

Quaternionic and Clifford operators are fundamental in mathematics and physics, appearing in quaternionic quantum mechanics \cite{Adler1995}, vector analysis, differential geometry, and hypercomplex analysis. \medskip

The application in \textit{vector analysis} is illustrated by the gradient operator with nonconstant coefficients, see also \cite{CMS24}. \medskip

In \textit{differential geometry}, the Dirac operator on a Riemannian manifold is expressed as
\begin{equation*}
\mathcal{D}=\sum_{i=1}^ne_i\nabla_{E_i}^\tau,
\end{equation*}
using covariant derivatives \cite{DiracHarm}. Specific examples include the Dirac operators in hyperbolic and spherical spaces, denoted as $\mathcal{D}_H$ and $\mathcal{D}_S$ respectively, which are explored in \cite{DIRACHYPSPHE}. \medskip

In \textit{hypercomplex analysis}, the standard Dirac operator $D$ and its conjugate $\overline{D}$ are extensively studied \cite{DSS,DiracHarm}. Furthermore, slice hyperholomorphic functions are connected to the kernel of the global operator $G$ introduced in \cite{6Global}. Recent developments include the Dirac fine structure on the $S$-spectrum, involving operators of the form
\begin{equation*}
T_{\alpha,m}=D^\alpha(D\overline{D})^m\qquad\text{and}\qquad\widetilde{T}_{\beta,m}=\overline{D}^\beta(D\overline{D})^m,
\end{equation*}
which link specific function spaces to functional calculi.

\end{document}